\documentclass[11pt]{amsart} 
\calclayout
\usepackage{amsfonts}
\DeclareMathAlphabet{\mathfrak}{U}{euf}{m}{n}

\usepackage{setspace,float} 
\usepackage{times}
\usepackage{euler}
\usepackage{adjustbox}
\usepackage{graphicx}
\usepackage{microtype}
\usepackage{pgfplots}
\pgfplotsset{compat=1.15}
\usepackage{mathrsfs}
\usepackage{hyperref}
\usepackage{mathrsfs}
\usepackage{comment}
\usepackage{wrapfig}
\usetikzlibrary{arrows}
\usepackage[mathcal]{eucal}

\usepackage[english]{babel}
\usepackage[doi=false,isbn=false,url=false,eprint=false,maxbibnames=5,minbibnames=4,backend=bibtex]{biblatex}
\usepackage{csquotes}
\usepackage{amssymb}
\usepackage{amsmath}
\usepackage{amsthm}
\usepackage{enumitem}
\newtheorem{thm}{Theorem}[section]
\newtheorem{lem}[thm]{Lemma}
\newtheorem{prop}[thm]{Proposition}
\newtheorem{coro}[thm]{Corollary}
\newtheorem*{intro}{Theorem}
\theoremstyle{remark}
\newtheorem{rema}[thm]{Remark}
\newtheorem{exa}[thm]{Example}
\newtheorem{defi}[thm]{Definition}

\DeclareFontFamily{U} {MnSymbolA}{}
\DeclareFontShape{U}{MnSymbolA}{m}{n}{
  <-6> MnSymbolA5
  <6-7> MnSymbolA6
  <7-8> MnSymbolA7
  <8-9> MnSymbolA8
  <9-10> MnSymbolA9
  <10-12> MnSymbolA10
  <12-> MnSymbolA12}{}
\DeclareFontShape{U}{MnSymbolA}{b}{n}{
  <-6> MnSymbolA-Bold5
  <6-7> MnSymbolA-Bold6
  <7-8> MnSymbolA-Bold7
  <8-9> MnSymbolA-Bold8
  <9-10> MnSymbolA-Bold9
  <10-12> MnSymbolA-Bold10
  <12-> MnSymbolA-Bold12}{}

\DeclareSymbolFont{MnSyA} {U} {MnSymbolA}{m}{n}
\DeclareMathSymbol{\lcirclearrowright}{\mathrel}{MnSyA}{252}
\DeclareMathSymbol{\rcirclearrowleft}{\mathrel}{MnSyA}{250}

\newcommand{\A}{\mathcal{A}}
\newcommand{\B}{\mathcal{B}}

\newcommand{\F}{\widetilde{F}_{2n}}
\newcommand{\ee}{\mathbf{e}}
\renewcommand{\a}{\alpha}

\newcommand{\sss}{\mathtt{s}}
\newcommand{\ttt}{\mathtt{t}}

\usepackage{stmaryrd}
\usepackage{slashed}
\usepackage{tikz-cd}
\tikzcdset{scale cd/.style={every label/.append style={scale=#1},
    cells={nodes={scale=#1}}}}
\usetikzlibrary{matrix}
\usetikzlibrary{positioning,hobby}
\usetikzlibrary{calc,intersections}
\usepackage{todonotes}
\usepackage[all,cmtip]{xy}
\title{Morita equivalences, moduli spaces and flag varieties}
\author{Daniel \'Alvarez}
\address{Instituto de Matem\'{a}tica Pura e Aplicada, Estrada Dona Castorina 110, Rio de
Janeiro, 22460-320, Brazil}
\email{uerbum@impa.br}

\date{} 
\begin{document}

\begin{abstract}  
Double Bruhat cells in a connected complex semisimple Lie group $G$ emerged as a crucial concept in the work of S. Fomin and A. Zelevinsky on total positivity and cluster algebras. These cells are special instances of a broader class of cluster varieties known as generalized double Bruhat cells, which can be studied collectively as Poisson subvarieties of $\F = \B^{2n-1} \times G$, where $\B$ is the flag variety of $G$. The spaces $\F$ are Poisson groupoids over $\B^n$ and were introduced by J.-H. Lu, V. Mouquin, and S. Yu in the study of configuration Poisson groupoids of flags.

In this work, we describe the spaces $\F$ as decorated moduli spaces of flat $G$-bundles over a disc. This perspective yields the following results:
\begin{enumerate}
    \item We explicitly integrate the Poisson groupoids $\F$ to symplectic double groupoids, which are complex algebraic varieties. Furthermore, we show that these integrations are symplectically Morita equivalent for all $n$.
    \item Using this construction, we integrate the Poisson subgroupoids of $\F$ formed by unions of generalized double Bruhat cells to explicit symplectic double groupoids. As a corollary, we obtain integrations for the top-dimensional generalized double Bruhat cells contained therein.
    \item Finally, we relate our integration to the work of P. Boalch on meromorphic connections. We lift the torus actions on $\F$ to the double groupoid level and show that they correspond to the quasi-Hamiltonian actions on the fission spaces of irregular singularities.
\end{enumerate} 
\end{abstract}
\maketitle
\tableofcontents 

\section*{Introduction}  It is a remarkable fact that many important Poisson structures, originally constructed using Lie theory, have later been rediscovered in the study of moduli spaces of flat connections. This novel interpretation often leads to a deeper understanding of these Poisson structures. For instance, the pioneering work in \cite{quahammer,boastopoi,focros,sevsomtit} recasts Poisson-Lie groups and their corresponding symplectic groupoids within this framework. 

This work is motivated by the study of the (generalized) double Bruhat cells associated with a connected complex semisimple Lie group $G$. These are Poisson varieties that are significant because of their connections to geometric representation theory and their appearance at the inception of the theory of cluster algebras. By interpreting these objects as moduli spaces, we associate symplectic groupoids to them, which can aid in understanding their quantizations and other applications, as discussed below. In fact, symplectic groupoids are central to Poisson geometry, serving as global counterparts to Poisson structures, much like Lie groups serve as global versions of Lie algebras. Beyond their direct applications to Poisson manifolds (for instance, when constructing normal forms \cite{fermar}), symplectic groupoids provide interesting ways to connect the classical and the quantum worlds \cite{braqua,focgonqua} and encapsulate generalized K\"ahler metrics at the global level \cite{intgk}.

In this context, $G$, equipped with its standard multiplicative Poisson structure depending on a pair of opposite Borel subgroups $(B_+, B_-)$, is a Poisson group with dual Poisson group $G^* = B_+ \times_T B_-$, where $T$ is the maximal torus $B_+ \cap B_-$. Let $W$ be the Weyl group of $(G, T)$. The double Bruhat cells $G^{u,v} = B_+uB_+ \cap B_-vB_- \subset G$ for $u,v \in W$ first emerged in the study of total positivity and cluster algebras \cite{bfz,doubru}. These cells were later found by M. Gekhtman, M. Shapiro, and A. Vainshtein to inherit a Poisson structure from $G$ which is compatible with the cluster algebra structure on their coordinate rings \cite{clualgpoigeo}. Furthermore, J.-H. Lu and V. Mouquin discovered that this Poisson structure can be naturally extended into either a Poisson groupoid or a Poisson bimodule structure, as pointed out in \cite{lumou2}. In fact, it is proven in \cite{lumou2} that the double Bruhat cell $G^{w,w}$ is an algebraic Poisson groupoid over the Schubert cell $B_+wB_+/B_+$, and $G^{v,w}$ functions as a bimodule for a pair of commuting Poisson actions of $G^{v,v}$ and $G^{w,w}$. Furthermore, the symplectic leaves of $G^{w,w}$ passing through the unit submanifold are symplectic groupoids. Collectively, double Bruhat cells can be studied as Poisson subgroupoids of the action groupoid
\[ G \ltimes G/B_+ \rightrightarrows G/B_+ \]
associated with the translation action of $G$ on the flag variety $\mathcal{B}= G/B_+$ \cite{lumou2,conpoigro}.

The Poisson groupoid structure on a double Bruhat cell was generalized in \cite{conpoigro} and described as an instance of a {\em configuration Poisson groupoid of flags}. In this broader framework, one considers a Poisson groupoid structure on 
\[ \widetilde{F}_{2n} := \overbrace{G \times_{B_+} G \times_{B_+} \dots \times_{B_+} G}^{2n \text{ times}} \rightrightarrows  F_n := \overbrace{G \times_{B_+} G \times_{B_+} \dots \times_{B_+} G/B_+}^{n \text{ times}} \cong (G/B_+)^n; \]
where the Poisson structure on each space is a {\em mixed product} as defined in \cite{lumou}. Generalized Schubert cells in $F_n$ play a role analogous to Schubert cells in $G/B_+$ and are given by
\[ \mathcal{O}^{\mathbf{u}} = B_+u_1B_+ \times_{B_+} \dots \times_{B_+} B_+u_nB_+/B_+, \quad u_i \in W. \]
We can identify $\widetilde{F}_{2n}\cong \mathcal{B}^{2n-1} \times G$ and view the {\em total configuration Poisson groupoid of flags} 
\[ \Gamma_{2n}=\mathcal{B}^{2n-1} \times \mathcal{A}^\circ \cong \mathcal{B}^{2n-1} \times B_- \subset \widetilde{F}_{2n}\] 
as a Poisson subgroupoid, where $\mathcal{A}^\circ=B_-N_+/N_+\cong B_-$ for $N_+\subset B_+$ the maximal unipotent subgroup. Then the generalized double Bruhat cells can be identified with the $T$-orbits of the symplectic leaves in $\Gamma_{2n}$ under a natural $T$-action by automorphisms, these are called {\em $T$-leaves} \cite{conpoigro}. Consequently, $\Gamma_{2n}$ decomposes as a finite disjoint union of such $T$-leaves, each of which becomes a cluster variety by viewing it as a special double Bott-Samelson cell \cite{cludoubot} (for $G$ simply-connected or of adjoint type).

As previously noted, integrating a Poisson manifold involves constructing a symplectic groupoid that induces the given Poisson structure. In the context of Poisson groupoids, the integration problem requires the construction of a {\em symplectic double groupoid}. This is a symplectic manifold that serves as the total space for two compatible groupoid structures, both of which are defined over Poisson groupoids that we say are in duality. We illustrate such objects as in the diagram below:  
\begin{align*}   \xymatrix{ \mathcal{G}  \ar@<-.5ex>[r] \ar@<.5ex>[r]\ar@<-.5ex>[d] \ar@<.5ex>[d]& \mathcal{K}^*   \ar@<-.5ex>[d] \ar@<.5ex>[d] \\ \mathcal{K}   \ar@<-.5ex>[r] \ar@<.5ex>[r] & M, } \label{eq:dousym} \end{align*}
so here $\mathcal{G} $ is a symplectic manifold equipped with two groupoid structures over the Poisson groupoids in duality $\mathcal{K}\rightrightarrows M $ and $\mathcal{K}^* \rightrightarrows M$. The vertical structure maps are groupoid morphisms with respect to the horizontal groupoid structures and viceversa.

This work is a natural continuation of \cite{poigromod,tracou,lumou2,lumou,conpoigro}, combining many insights therein to solve the integration problem of the Poisson groupoids $\widetilde{F}_{2n} \rightrightarrows F_n$, viewed as complex algebraic Poisson varieties. Perhaps surprisingly, we show that the integrations we construct are Morita equivalent. 
\begin{intro} The Poisson groupoid $\widetilde{F}_{2n} \rightrightarrows F_n$ is integrable by an explicit complex algebraic symplectic double groupoid $\mathcal{G}_{2n}$ with dual Poisson groupoid $\widetilde{F}_{2n}^* \rightrightarrows F_n$. Moreover, there is a distinguished symplectic Morita equivalence $\mathcal{G}_{n+m} $ between $\mathcal{G}_{2n} $ and $\mathcal{G}_{2m}  $ for all $n,m$ positive integers. \end{intro} 
Let us point out that the Morita equivalences appearing in this Theorem take into account the double groupoid structures as in \cite{intgk}, see Proposition \ref{pro:symdou} and Theorem \ref{thm:symdoumor} below. Our integration also shows that generalized Schubert cells are orbits of the dual Poisson groupoid $\widetilde{F}_{2n}^* \rightrightarrows F_n$ and hence they are contained in symplectic leaves of $\widetilde{F}_{2n} \rightrightarrows F_n$, thus recovering an observation in \cite{conpoigro}, see Corollary \ref{cor:gensch}.
 
Furthermore, we obtain integrations of the total configuration Poisson groupoids of flags and, consequently, of the open generalized double Bruhat cells contained therein. Specifically, we establish that $\Gamma_{2n} \subset \widetilde{F}_{2n}$ is a \textit{Lie-Dirac submanifold} \cite{craruipoi, dirsub} which admits an integration realized as a symplectic double subgroupoid of $\mathcal{G}_{2n}$. This discussion, along with the results in Theorems \ref{thm:inttotcon} and \ref{thm:resmorequ}, is summarized as follows.
\begin{intro} 
The Poisson groupoid $\Gamma_{2n} \rightrightarrows F_n$ is integrable by a symplectic double subgroupoid $\mathcal{H}_{2n} \subset \mathcal{G}_{2n}$. Furthermore, the symplectic Morita equivalence $\mathcal{G}_{n+m}$ between $\mathcal{G}_{2n}$ and $\mathcal{G}_{2m}$ restricts to a symplectic Morita equivalence between $\mathcal{H}_{2n}$ and $\mathcal{H}_{2m}$ for all positive integers $n$ and $m$. 
\end{intro}
Finally, by lifting the natural $T$-action on $\widetilde{F}_{2n}$ to its integration, we obtain the quasi-hamiltonian action on the moduli space of meromorphic connections on a disc with a pole of order two, this is one of the simplest (higher) fission spaces introduced in \cite{geobrasto,quahammer,fisspa}. See Theorem \ref{thm:lifact} for this result, further ramifications of this interpretation and applications to the study of integrable systems as in \cite{kogzel} shall be treated elsewhere.
\begin{intro} The natural $T$-action on $\widetilde{F}_{2n}$ lifts to a quasi-hamiltonian $T$-action on $\mathcal{G}_{2n}$. \end{intro}  

Our approach relies on the versatile technique of constructing double groupoid structures and Morita equivalences through decorated moduli spaces of flat bundles over surfaces, introduced in \cite{sevmorqua} and further developed in \cite{poigromod,intgk}. This construction builds upon the foundational work of \cite{alemeimal,alemeikos} which, in its most general form (as in \cite{quisur,quisur2}), encompasses the moduli spaces of meromorphic connections studied in \cite{geobrasto,quahammer,fisspa}. To prove our results, we refine the methods in \cite{poigromod} to handle Morita equivalences and adapt the study of Poisson submanifolds and their integrability from \cite{craruipoi} to our context. In particular, we show that any decorated moduli space is a Poisson submanifold of an explicitly integrable Poisson manifold. 

In future work, we will investigate the striking parallelism between our construction of symplectic groupoids via surface gluing and the {\em symplectic double} of a cluster variety introduced by V. V. Fock and A. B. Goncharov. This construction yields a symplectic groupoid as a semi-classical limit of a {\em quantum double} \cite{focgonqua,gonshe} (see also \cite{symgrocluman}). The resemblance is particularly evident because moduli spaces of flat bundles provide key examples of cluster varieties, and their symplectic doubles have been studied in \cite{symdoumodspa}.  
Finally, let us mention that the symplectic double groupoids we use may be extended to symplectic 2-groupoids that integrate the underlying Courant algebroids, this will be the subject of a separate work about integration of action Courant algebroids \cite{sym2gro}.

\subsection*{Acknowledgements} The author is very grateful to Jiang-Hua Lu for posing the question of integrating $\widetilde{F}_{2n}$ using decorated moduli spaces and to Henrique Bursztyn for teaching him the key ideas needed in this endeavor. Numerous stimulating conversations with Francis Bischoff, Alejandro Cabrera, Miquel Cueca, Marco Gualtieri and Eckhard Meinrenken have further influenced this work. The author is also grateful to the referee for many useful comments and suggestions.
\section{Preliminaries} In this paper, we work with Poisson structures on complex algebraic varieties that can be most efficiently constructed using the language of quasi-Poisson manifolds. See \cite{poigromod} and its references for a thorough review of this theory and its reformulation in terms of Dirac geometry; \S \ref{subsec:quapoimodspa} briefly recalls the formalism of quasi-Poisson manifolds and representation varieties introduced in \cite{alemeimal,alemeikos}, see also \cite{quisur,quisur2} and specially \cite{geobrasto,quahammer} for fundamental examples in the complex category. 

\subsection{Quasi-Poisson structures on spaces of representations and boundary decorations}\label{subsec:quapoimodspa} Throughout this paper we will be dealing with a connected real or complex Lie group $G$ such that its Lie algebra $\mathfrak{g} $ is equipped with a nondegenerate symmetric Ad-invariant bilinear form $\langle \,,\,  \rangle $ (which is $\mathbb{C}$-bilinear if $G$ is a complex Lie group). Also, we denote by $(\cdot)^\sharp$ the image of an element under the isomorphism $\mathfrak{g}^* \rightarrow \mathfrak{g} $ defined by the pairing $\langle \,,\,  \rangle $. The following definitions make sense in either the smooth or the holomorphic setting, we omit the corresponding adjectives unless they are needed. A \emph{quasi-Poisson manifold} \cite{alemeikos} is a manifold $M$ together with $\pi\in \mathfrak{X}^2(M)$ and an action $\rho:\mathfrak{g}\to TM$ satisfying
 \[ \mathcal{L}_{\rho(u)}\pi=0 \quad \text{and}\quad \frac{1}{2} [\pi,\pi]=\rho(\chi)\quad  \forall u\in \mathfrak{g}; \]
where $\chi\in \wedge^3 \mathfrak{g} $ is defined as $\chi(\xi,\eta,\zeta)=\frac{1}{4}\langle \xi,[\eta^\sharp,\zeta^\sharp] \rangle  $ for all $\xi,\eta,\zeta\in \mathfrak{g}^* $.

The key examples of quasi-Poisson manifolds are the following objects. Let $\Sigma$ be a compact and oriented surface such that each of its components has nonempty boundary and let $V\subset \partial \Sigma$ be a finite set that meets every component of $\Sigma$. In this situation, we call $(\Sigma,V)$ a {\em marked surface}. The flat $G$-bundles over $\Sigma$ equipped with a fixed trivialization over $V$ may be parametrized by the manifold of groupoid morphisms $\hom(\Pi_1(\Sigma,V),G)$, where $\Pi_1(\Sigma,V)\rightrightarrows V$ is the fundamental groupoid of $\Sigma$ based at $V$. 
 
Following \cite{alemeikos,quisur}, we shall describe a distinguished quasi-Poisson structure on $M=\hom(\Pi_1(\Sigma,V),G)$ whose underlying action is the infinitesimal counterpart of the gauge action of $G^V$ on $M$ given by 
\begin{align} (g\cdot \rho)(\gamma )=g_{\mathtt{t}(\gamma)}\rho(\gamma)g_{\mathtt{s}(\gamma)}^{-1}\quad  \forall \rho \in \hom(\pi_1(\Sigma,V),G),\, g\in G^V. \label{eq:gauact} \end{align}   
The corresponding bivector field is denoted by $\pi_{\Sigma,V}$. To describe this bivector field explicitly we use a {\em skeleton} for $(\Sigma,V)$, which is an embedded oriented 1-dimensional CW-complex $\mathfrak{S}  \subset \Sigma$ that satisfies: \begin{itemize}
    \item its set of vertices (0-dimensional cells) is $\mathfrak{S}_0= V$ and
    \item $\Sigma $ retracts to $\mathfrak{S} $ by deformation.
\end{itemize}  
In what follows we view $\mathfrak{S} $ as an oriented graph $(\mathfrak{S}_1,\mathfrak{S}_0)$, where $\mathfrak{S}_1 $ is its set of edges (1-cells) and $\mathtt{S}(\mathbf{e} ), \mathtt{T}(\mathbf{e})\in \mathfrak{S}_0 $ denote the respective source and target of $\mathbf{e} \in \mathfrak{S}_1$. The choice of $\mathfrak{S}$ allows us to identify $M\cong G^{\mathfrak{S}_1 }$. We denote by $Y(\mathbf{e})$ for $Y\in \mathfrak{X}(G)$ the vector field on $M$ which is $Y$ on the factor corresponding to $\mathbf{e} \in \mathfrak{S}_1 $ and $0$ otherwise. For each $v\in V$, the orientation of $\Sigma$ determines a total order on the set of half-edges $\mathfrak{S}(v)$ adjacent to $v$: 
\[ \mathfrak{S}(v)=\{ ({\bf e} ,v,\sigma)\ |\ \text{$\sigma\in\{\mathtt{S},\mathtt{T} \}$ and $\sigma(\mathbf{e})=v$}  \}.\]
One takes the outward direction to $\partial \Sigma$ as reference and then one orders the half-edges in $\mathfrak{S}(v) $ counterclockwise; we denote by $h <h'$ the order relation so established. Let $\{X_i\}_i\subset \mathfrak{g} $ be a basis with $ s_{ij}=\langle X_i,X_j \rangle $ and denote by $s^{ij}$ the coefficients of the matrix inverse to $(s_{ij})_{ij}$. Then the quasi-Poisson bivector field on $M$ is as follows:
\begin{align}  \pi_{\Sigma,V}=-\frac{1}{2} \sum_{v\in V}\sum_{\substack{h<h'\\ h,h'\in \mathfrak{S}(v)}} \sum_{i,j} s^{ij} X_i(h)\wedge X_j(h'), \qquad
 X_i(h)=\begin{cases} -X_i^r(\mathbf{e}), \quad &\text{if $h=(\mathbf{e}, v,\mathtt{T})$}, \\
X_i^l(\mathbf{e}), \quad &\text{if $h=(\mathbf{e}, v,\mathtt{S}) $}; \end{cases} \label{eq:qpoimodspa} \end{align} 
where $X^l,X^r$ are the respective left-invariant and right-invariant vector fields on $G$ determined by $X\in \mathfrak{g} $, see \cite[Equation 14]{quisur}. 
\begin{rema} In the presence of a classical $r$-matrix $r$ on $\mathfrak{g} $ with symmetric part equal to the pairing $\langle\ ,\, \rangle $, the bivector field in \eqref{eq:qpoimodspa} is closely related to the Fock-Rosly Poisson structure \cite{focros}, see \cite[Remark 6]{quisur} and \cite{moufoc} for a detailed comparison. We follow the sign convention used in \cite{quisur2} for \eqref{eq:qpoimodspa}. \end{rema}

\subsubsection{Moment maps and boundary decorations} The induced orientation on the boundary $\partial \Sigma$ of a marked surface $(\Sigma,V)$ serves to define a {\em boundary graph $\Gamma=(E,V) $} in which $V$ is the set of vertices and the set of edges $E$ consists of the negatively oriented boundary arcs in $\partial \Sigma$ lying between pairs of consecutive points in $V$; we follow \cite{quisur} in its convention for this choice of orientation. We shall continue using the notation $M=\hom(\Pi_1(\Sigma,V),G)$ in what follows.

The significance of such a boundary graph is that it serves to define a moment map $\mu: M \rightarrow G^E $ given by $\mu(\rho)=(\rho(e))_{e\in E}$ which is $G^V$-equivariant with respect to \eqref{eq:gauact} and the $G^V$-action on $G^E$ defined by a similar expression
\[ g\cdot (a_e )_{e\in E}=(g_{\mathtt{T}(e)}a_eg_{\mathtt{S}(e)}^{-1} )_{e\in E}. \]Such a moment map makes $M$ into a hamiltonian quasi-Poisson $G^V$-manifold, see \cite{poigromod,quisur2} for details. The key fact for us is that we may use it to perform moment map reduction and thus obtain Poisson quotients. 
\begin{rema}\label{rem:infact} In order to identify $G^V\cong G^E$ we use the identification $E\cong V$ given by $e\mapsto \mathtt{T}(e)$ for all $e\in E$. So if $\partial \Sigma$ consists of a single component, we shall denote the elements of $V$ by $v_i$ and the elements of $E$ as $e_i$ in such a way that $v_i=\mathtt{T}(e_i)$ and $\mathtt{S} (e_i)=v_{i+1}$, where $i=1\dots n$ and $n$ is the number of elements in $V$. In this situation, the corresponding gauge $G^n$-action is
\begin{align} (g_1,\dots,g_n)\cdot (a_1,\dots,a_n)=(g_1a_1g_{2}^{-1},g_{2}a_{2}g_{3}^{-1},\dots,g_na_ng_1^{-1}), \label{eq:gauactdisc} \end{align}  
for all $g_i,a_i\in G$ and $i=1\dots n$.   \end{rema}

Let $\mathfrak{d}=\mathfrak{g} \oplus \mathfrak{g}  $ be the direct sum Lie algebra equipped with the pairing $\langle \ ,\ \rangle $ on the first factor and $-\langle \ , \ \rangle $ on the second one. We say that a subspace $\mathfrak{u}  \subset \mathfrak{d} $ is {\em lagrangian} if $\mathfrak{u} $ is maximally isotropic with respect to this pairing. For example, the diagonal $\mathfrak{g}_\Delta \subset \mathfrak{d}$ is a lagrangian Lie subalgebra and it has a distinguished lagrangian complement $\mathfrak{g}_{-\Delta}=\{u\oplus-u\ |\ u\in \mathfrak{g} \} $. So we may identify $\mathfrak{d} $ with $ \mathfrak{g}_\Delta\oplus \mathfrak{g}_{-\Delta}$ as vector spaces. Furthermore, the pairing gives us another identification $\mathfrak{g}_{-\Delta}\cong \mathfrak{g}^*\cong \mathfrak{g}_\Delta^*$ defined by $u\oplus -u\mapsto 2 \langle u,\ \rangle  $ for all $u\in \mathfrak{g} $. We shall freely use the vector space isomorphism $\mathfrak{d}\cong \mathfrak{g}_{\Delta}\oplus \mathfrak{g}_{\Delta}^*$ thus obtained in what follows.
  
A {\em decoration} of the marked surface $(\Sigma,V)$ is a pair $(H,\A)$, where $H\subset G^E$ is a submanifold and $\A=\prod_{v\in V} \mathfrak{a}_v \subset \mathfrak{d}^V $ is a Lie algebra consisting of a lagrangian Lie subalgebra $\mathfrak{a}_v\subset \mathfrak{d}  $ at each vertex, with the property that the $\A$-action on $G^E$: 
\begin{align*} \mathcal{A}\rightarrow \mathfrak{X}(G^E),\qquad  (X_v\oplus Y_v)_{v\in V}\mapsto (Y_{\mathtt{T}(e)}^r-X_{\mathtt{S}(e)}^l)_{e\in E}\in \mathfrak{X}(G^E),    \end{align*}  
preserves $H$; in other words, for each $h\in H$, the image of the action above is contained in $T_hH$. Our convention in the expression above is that the Lie bracket on a Lie algebra is defined using right invariant vector fields. Let $\mathfrak{g}_\Delta\subset \mathfrak{d}$ be the diagonal inclusion and let $K_v \subset G$ be the connected subgroup integrating $\mathfrak{a}_v\cap \mathfrak{g}_\Delta \subset\mathfrak{g}_\Delta $. Suppose that $\mu$ intersects $H$ cleanly and the action of $K=\prod_{v\in V} K_v \hookrightarrow G^V$ restricted to $\mu^{-1}(H)$ is free and proper, then \cite[Thm. 1.1]{quisur2} implies that there is a canonical Poisson structure on the {\em decorated moduli space} determined by the decoration $(H,\mathcal{A} )$:  
\begin{align} \mathfrak{M}_G(\Sigma,V)_{H,\mathcal{A}}:= \mu^{-1}(H)/K; \label{eq:decmodspa} \end{align}   
This decorated moduli space is symplectic if $V$ intersects every component of $\partial \Sigma$ and if $TH$ is spanned by the $\mathcal{A}$-action on $G^E$. Let $\mathtt{a}_{\Sigma,V}:\mathfrak{g}^V \rightarrow M$ be the infinitesimal counterpart of the gauge action \eqref{eq:gauact}. As shown in \cite[Thm. 1.4]{quisur2}, the Poisson tensor on \eqref{eq:decmodspa} is given by reducing 
\begin{align} \pi_{{\Sigma},{V}  }+\sum_{v\in {V}} \mathtt{a}_{\Sigma,V} (\pi_v) \in \Gamma \left(\wedge^2\left( T\mu^{-1}(H)/ \mathtt{a}_{\Sigma,V}\left(\bigoplus_{v\in V} \mathfrak{g}_\Delta \cap \mathfrak{a}_v   \right)\right) \right),\label{eq:bivmodspa} \end{align}  
where $\pi_v\in \wedge^2 \left(\mathfrak{g}_\Delta/(\mathfrak{g}_\Delta\cap \mathfrak{a}_v )\right)$ is the bivector resulting from dividing the lagrangian subspace $\mathfrak{a}_v\subset \mathfrak{d}\cong \mathfrak{g}_\Delta\oplus \mathfrak{g}_\Delta^* $ by its kernel $\mathfrak{a}_v\cap \mathfrak{g}_\Delta$ as in \cite[Prop. 1.1.4]{coudir}. In the symplectic case, one may describe the symplectic form on $\mathfrak{M}_G(\Sigma,V)_{H,\mathcal{A}}$ directly as the reduction of  a {\em quasi-symplectic 2-form $\Omega_{\Sigma,V}$} on $M$, see \S \ref{subsec:mom map surf}.

It is automatic that $\mathfrak{M}_G(\Sigma,V)_{H,\mathcal{A}}$ is a holomorphic Poisson or symplectic manifold, as long as $G$ is a complex Lie group, $H\subset G^E$ is a complex submanifold and $\mathcal{A}$ is given by complex Lie subalgebras. In the algebraic category, since the $K$-invariant functions on $\mu^{-1}(H)$ form a Poisson algebra, $ \mathfrak{M}_G(\Sigma,V)_{H,\mathcal{A}}$ is a complex algebraic Poisson or symplectic variety, if the quotient $\mu^{-1}(H) \rightarrow \mu^{-1}(H)/K$ is a geometric quotient in the sense of Mumford \cite{mumgit}. 
In the cases that interest us in this work, we obtain geometric quotients because they are given by associated bundles to flag varieties, see Proposition \ref{pro:slim}.

\subsection{Poisson groupoids from symmetric decorations}\label{subsec:poigromod} A Lie groupoid is a groupoid object in the category of smooth manifolds (or complex manifolds or smooth complex algebraic varieties) such that its source map is a submersion. So a Lie groupoid consists of a space of arrows $\mathcal{G} $ and a space of unit elements $M$, denoted together by $\mathcal{G} \rightrightarrows M$. The structure maps are its source, target, multiplication, inversion and unit inclusion maps, respectively denoted by $\mathtt{s}$, $\mathtt{t}$, $\mathtt{m}$ , $\mathtt{i}$ and $\mathtt{u}$ in what follows. To make the notation simpler, we shall also abbreviate $\mathtt{m}(a,b)=ab$ and $\mathtt{i}(g)=g^{-1}$, our convention here is that $\mathtt{s}(a)=\mathtt{t} (b)$. A Poisson groupoid is a Lie groupoid $\mathcal{G} \rightrightarrows M$ equipped with a Poisson structure on $\mathcal{G} $ such that the graph of the multiplication is coisotropic $\text{Graph}(\mathtt{m})=\{(x,y,\mathtt{m}(x,y) )\}\subset \mathcal{G} \times \mathcal{G} \times \overline{\mathcal{G} }    $, where $\overline{\mathcal{G} }$ denotes $\mathcal{G}$ equipped with the opposite Poisson structure \cite{weicoi}. Decorated moduli spaces provide a large source of examples of these objects \cite{poigromod}. 

Take two copies of $\Sigma$ and denote by $i_a:\Sigma \hookrightarrow \Sigma \sqcup \Sigma$ the inclusions for $a=1,2$. The two surfaces can be identified along a collection of arcs $S\subset E$ thus producing a new surface $\widehat{\Sigma}=\Sigma\cup_{S} \Sigma$:
\[ \Sigma\cup_{S} \Sigma= \Sigma\sqcup \Sigma/\sim, \quad i_1(x)\sim i_2(x)\quad  \forall x \in e,\, \forall e\in S. \] 
The orientation of $\widehat{\Sigma}$, is obtained by taking the original orientation on $\Sigma\cong i_1(\Sigma)$ and reversing the orientation of $i_2(\Sigma)$. Consider a set of vertices
\[ \widehat{V}\subset \left[i_1(V)\cup i_2(V) \right]\cap \partial\widehat{\Sigma}. \] 
As long as $(\widehat{\Sigma},\widehat{V}) $ is a marked surface and we choose a decoration $(H,\A)$ which is symmetric with respect to the natural reflection symmetry defined by the doubling (and so, in particular, $\widehat{V}$ should be symmetric):
\[ i_1(x) \mapsto i_2(x), \qquad i_2(x) \mapsto i_1(x), \qquad \forall x\in \Sigma, \] 
the corresponding decorated moduli space $\mathfrak{M}_G(\widehat{\Sigma},\widehat{V})_{H,\mathcal{A}}$ admits a natural Poisson groupoid structure \cite[Thm. 3.5]{poigromod}. The Lie groupoid structure on $\mathfrak{M}_G(\widehat{\Sigma},\widehat{V})_{H,\mathcal{A}}$ descends from a Lie groupoid structure on $\hom(\pi_1((\widehat{\Sigma},\widehat{V}),G)$ over $\hom(\pi_1(\Sigma,V),G)$ that can be described as follows:
\begin{itemize}
    \item restricting a flat $G$-bundle with a framing over ${\widehat{V} }$ to each of the images of the inclusions $i_j(\Sigma,V)\subset (\widehat{\Sigma},\widehat{V})$ defines the source and target maps for $j=1,2$:
    \[ \sss(P)=i_2^*P, \quad \ttt(P)=i_1^*P; \]
    \item two flat framed $G$-bundles $P,P'$ over $(\widehat{\Sigma},\widehat{V} )$ that satisfy $i_2^*P \cong i_1^*P'$ can be glued to a flat framed $G$-bundle $P\circ P'$ over $(\widehat{\Sigma},\widehat{V} )$ such that $i_1^*(P\circ P')\cong i_1^*P$ and $i_2^*(P\circ P')\cong i_2^*P'$, this operation determines a multiplication.
\end{itemize}
The canonical moment map $ \mu:\hom(\pi_1(\widehat{\Sigma},\widehat{V} ),G)\rightarrow  G^{\widehat{ E}}$ is a Lie groupoid morphism with respect to the Lie groupoid structure we have just described, see \cite[Prop. 3.11]{poigromod}. Hence $\mathfrak{M}_G(\widehat{\Sigma},\widehat{V})_{H,\mathcal{A}}$ arises from multiplicative moment map reduction. See \S\ref{subsec:expgro} for explicit fomulae in the cases that are relevant for us and a more detailed review of this general groupoid structure in \S\ref{subsec:mor}. 

\section{Configuration Poisson groupoids of flags as decorated moduli spaces}\label{sec:confla} 
Now we recall the construction of the Poisson groupoids over products of flag varieties defined in \cite{conpoigro} and we put them into the framework of decorated moduli spaces \cite{quisur,quisur2}, thus setting the stage for their explicit integration according to \cite{poigromod}.
\subsection{The standard Poisson-Lie structure on complex semi-simple Lie groups}\label{subsec:st poissson} Let $G$ be a connected complex semi-simple Lie group. Equip $\mathfrak{g} $ with a nondegenerate Ad-invariant symmetric bilinear form $\langle \,,\,  \rangle $ and choose a pair of opposite Borel subgroups $B_+,B_-$ in $G$. Let $\mathfrak{t}$ be the Lie algebra of the maximal torus $T=B_+\cap B_-$. Let $\Delta_+\subset \mathfrak{t}^* $ be the set of positive roots associated to $\mathfrak{t} $. For each $\alpha\in  \Delta_+$, let $E_{\alpha }\in \mathfrak{g}_{\alpha } $ and $E_{-\alpha }\in \mathfrak{g}_{-\alpha }$ be such that $\langle E_{\alpha },E_{-\alpha}  \rangle=1$. The {\em standard quasitriangular $r$-matrix} associated to $\mathfrak{t} $ is
\[ r=\frac{1}{2} \sum_i h_i\otimes h_i +\sum_{\alpha \in \Delta_+} E_{-\alpha } \otimes E_{\alpha }, \]  
where $\{h_i\}_i$ is an orthonormal basis of $\mathfrak{t} $. The {\em standard Poisson-Lie group structure on $G$} is then given by $\pi_G=\Lambda^r-\Lambda^l$, where $\Lambda$ is the skew-symmetric part of $r$. We denote by $N_\pm\subset B_\pm $ the corresponding maximal unipotent subgroups with Lie algebras $\mathfrak{n}_\pm$ and we let $\text{pr}_{\mathfrak{t} }:\mathfrak{b}_\pm=\mathfrak{n}_\pm\oplus \mathfrak{t}   \rightarrow \mathfrak{t} $ be the projections. The Manin triple corresponding to $(G,\pi_G)$ is given by the direct sum $\mathfrak{d}=\mathfrak{g} \oplus \mathfrak{g} $ as the double with the pairing $\text{pr}_1^*\langle \,,\, \rangle - \text{pr}_2^* \langle \,,\, \rangle $, $\mathfrak{g} $ sits inside it as the diagonal lagrangian Lie subalgebra and we may identify ${\mathfrak{g} }^*=\{X\oplus Y\in  \mathfrak{b}_+\oplus \mathfrak{b}_-\ |\ \text{pr}_{\mathfrak{t}}(X)=- \text{pr}_{\mathfrak{t}}(Y)\}$ with the dual Lie algebra.

\subsection{Poisson groupoids over products of flag varieties}\label{subsec:poigrofla} There is an interesting family of Poisson groupoids associated to the standard Poisson group structure on $G$ which was introduced in \cite{conpoigro}. Take the Poisson manifold 
\[ \widetilde{F}_n :=\overbrace{G \times_{B_+}G \times_{B_+} \dots \times_{B_+} G}^{n \text{ times}} \]   
which is the quotient of $G^{n}$ by the $B_+^{n-1}$-action
\[ (b_{1},\dots,b_{n-1})\cdot (g_1,\dots,g_{n}) =(g_1b_{1}^{-1},b_{1}g_{2}b_{2}^{-1},\dots,b_{n-2}g_{n-1}b_{n-1}^{-1},b_{n-1}g_n) \]
for all $b_i\in B_+$ and all $g_i\in G$. The Poisson structure on $\widetilde{F}_n$ is induced by the standard Poisson group structure on $G$ and the fact that $B_+ \subset G$ is a Poisson subgroup, this fact implies that there is a Poisson structure on the quotient $\widetilde{F}_n\cong G^n/B_+^{n-1}$. Consider the residual left $B_+$-action on the last factor of the above product $b\cdot [g_1,\dots,g_n]=[g_1,\dots,g_nb^{-1}]$ and denote its quotient as $ F_n:=B_+\backslash\widetilde{F}_n  $. Then we have that the quotient by the diagonal action $B_+\backslash (\widetilde{F}_{n} \times \overline{\widetilde{F}_n}) $ is isomorphic as a Poisson manifold to $\widetilde{F}_{2n}$ (see Theorem \ref{pro:conflaqpoi} below) and it is a Poisson groupoid over $F_n$. Furthermore, we have that $\widetilde{F}_{n+m}$ serves as a natural bimodule for a Morita equivalence of Poisson groupoids between $\widetilde{F}_{ 2n}$ and $\widetilde{F}_{2m} $. 
\[\begin{tikzcd}
\widetilde{F}_{ 2n} \arrow[d, shift right] \arrow[d, shift left] &  & \widetilde{F}_{n+m}  \arrow[lld] \arrow[rrd] &  & \widetilde{F}_{ 2m} \arrow[d, shift right] \arrow[d, shift left] \\
F_n                                                              &  &                                             &  & F_m.                                                             
\end{tikzcd}\]
We will give more details in Theorem \ref{thm:symdoumor}, where we upgrade this observation to a Morita equivalence of symplectic double groupoids.

\begin{wrapfigure}{r}{5cm}\begin{tikzpicture}[line cap=round,line join=round,>=stealth,x=1cm,y=1cm,scale=0.6]
\clip(6.025310782816655,3.939691237049708) rectangle (15.156000919477579,10.43087314881562);
\draw [line width=1.2pt] (10.038525092765054,7.201928556518199) circle (2.538970219326206cm);

\draw [line width=1pt,dashed,->] (11.536924420615055,9.251603952507534)-- (8.541048363676854,9.25227809268628); 
\draw [line width=1pt,dashed,->] (7.697106682724682,6.2200275412679975)-- (8.541048363676854,9.25227809268628); 
\draw [line width=1pt,dashed,->] (9.951549208037106,4.664448513226991)-- (8.541048363676854,9.25227809268628);   
\draw [line width=1pt,dashed,->] (12.400823407913643,6.271382241243403)-- (8.541048363676854,9.25227809268628);  

\begin{scriptsize}
\draw [fill=black] (11.536924420615055,9.251603952507534) circle (1.2pt);
\draw[color=black] (11.612748030624086,9.530298699432917) node {$v_5$};
\draw [fill=black] (8.541048363676854,9.25227809268628) circle (1.2pt);
\draw[color=black] (8.428123976674794,9.537471276131001) node {$v_1$};

\draw[color=black] (10.063471463837946,8.9) node {$\ee_4$};
\draw [fill=black] (9.951549208037106,4.664448513226991) circle (1.2pt);
\draw[color=black] (9.6,4.9) node {$v_3$};
\draw [fill=black] (7.697106682724682,6.2200275412679975) circle (1.2pt);
\draw[color=black] (7.3,6.396298756143437) node {$v_2$};
\draw [fill=black] (12.400823407913643,6.271382241243403) circle (1.2pt);
\draw[color=black] (12.88062981109225,6.446506793030024) node {$v_4$};

\draw[color=black] (7.9,7.8) node {$\ee_1$};
\draw[color=black] (9.0,6.8) node {$\ee_2$};
\draw[color=black] (10.5,7.4) node {$\ee_3$};
\end{scriptsize}
\end{tikzpicture} \caption{A skeleton of the disc $(\Sigma_4,V_4)$}
  \label{fig:dis}\end{wrapfigure}
 
Another description of the Poisson structure on $F_n$ is as a {\em mixed product} of $n$ copies of $G/B_+$ in the sense of \cite{lumou}, where the flag variety $G/B_+$ is viewed as a Poisson homogeneous space for $G$, see \cite[\S A.2]{conpoigro}.
\subsection{Decorated discs and configuration Poisson groupoids of flags}\label{subsec:decdis} While the construction above relies on quotients of Poisson actions, there is an equivalent quasi-Poisson viewpoint that seems more convenient for the purpose of integration. Let $(\Sigma_k,V_k)$ be the marked surface given by the disc with $k+1$ marked points $v_i$ ordered cyclically on its boundary for $k\geq 2$, i.e. $V_k=\{v_1,\dots,v_{k+1}\}$ and consider the associated space of representations $M_{k}:=\hom(\Pi_1(\Sigma_k,V_k),G)$. We may use the oriented boundary edges defined by $V_k$ as generators of the fundamental groupoid $\Pi_1(\Sigma_k,V_k)$ to identify $M_{k}=\{(a_1,\dots,a_{k+1})\in G^{k+1}\ |\ a_1a_2\dots a_{k+1}=1\}$ and thus the gauge action \eqref{eq:gauact} is as in \eqref{eq:gauactdisc} and the moment map is the inclusion $M_k\subset G^{k+1}$.
Consider the boundary decoration $(H_{k},\mathcal{A}_{k}) $ of $(\Sigma_k,V_k)$ given by decorating the vertices $v_i$ as follows
\begin{align}  {\mathfrak{a} }_{v_i}=\begin{cases} &\widetilde{\mathfrak{b} }_+=\{X\oplus Y\in \mathfrak{b}_+\oplus \mathfrak{b}_+\ | \ X-Y\in \mathfrak{n}_+   \}, \quad \text{if $2\leq i\leq k $}  \\
&{\mathfrak{g} }^*=\{X\oplus Y\in \mathfrak{b}_+\oplus \mathfrak{b}_-\ |\ \text{pr}_{\mathfrak{t}}(X)=- \text{pr}_{\mathfrak{t}}(Y)\}, \quad \text{if $i=1$,} \\
&({\mathfrak{g} }^*)^\vee=\{X\oplus Y\in  \mathfrak{b}_-\oplus \mathfrak{b}_+\ |\ \text{pr}_{\mathfrak{t}}(X)=- \text{pr}_{\mathfrak{t}}(Y)\}, \quad \text{if $i=k+1$;} \end{cases}   \label{eq:surdec}      \end{align}
and all the edges are decorated by $G$ itself so that $H_k=G^{k+1}$. According to this prescription, the endpoints of the edge labelled by $g_{k+1}$ are decorated with copies of $\mathfrak{g}^* $, whereas all the other vertices are decorated with $\widetilde{\mathfrak{b} }_+ $. Note that, by construction, the group $K$ appearing in \eqref{eq:decmodspa} is $B_+^{k-1}$ and hence $\mathfrak{M}_G(\Sigma_k,V_k)_{H_{k},\mathcal{A}_{k}} =M_k/B_+^{k-1}$, where the action is \eqref{eq:gauactdisc} restricted to $\{1\} \times B_+^{k-1}\times \{1\}\subset G^{k+1} $.
\begin{thm}\label{pro:conflaqpoi} The decorated moduli space $\mathfrak{M}_G(\Sigma_k,V_k)_{H_{k},\mathcal{A}_{k}} $ is Poisson isomorphic to $\widetilde{F}_{k} $ by means of the map
\begin{align*}  &\Psi_{k}:\mathfrak{M}_G(\Sigma_k,V_k)_{H_{k},\mathcal{A}_{k}} \rightarrow \widetilde{F}_{k},\\
& [a_i]_{i=1\dots k+1}\mapsto [a_1,a_2,\dots,a_{k-1},a_{k}],\quad \text{$[a_i]_{i=1\dots k+1}\in \mathfrak{M}_G(\Sigma_k,V_k)_{H_{k},\mathcal{A}_{k}}$.}    \end{align*} 
 \end{thm}  
\begin{rema}\label{rem:poimap} By putting $k=2n$, we obtain, in particular, that the following map is an isomorphism:
\begin{align*}  &\Psi'_n:\mathfrak{M}_G(\Sigma_{2n},V_{2n})_{H_{2n},\mathcal{A}_{2n}} \rightarrow B_+\backslash (\widetilde{F}_{n} \times \overline{\widetilde{F}_n}) \cong \widetilde{F}_{2n} ,\\ 
&[a_i]_{i=1\dots 2n+1}\mapsto [(a_1,a_{2},\dots,a_{n}),(a_{2n}^{-1},a_{2n-1}^{-1},\dots,a_{n+1}^{-1})], \quad \text{$[a_i]_{i=1\dots 2n+1}\in \mathfrak{M}_G(\Sigma_{2n},V_{2n})_{H_{2n},\mathcal{A}_{2n}}$.}   \end{align*} 
In fact, the following map is a Poisson isomorphism
\begin{align*}  &\chi:B_+\backslash(\widetilde{F}_{n} \times \overline{\widetilde{F}_n}) \rightarrow \widetilde{F}_{2n} =\overbrace{G \times_{B_+} G \times_{B_+} \dots \times_{B_+} G}^{2n \text{ times}} \\
&\chi([(g_{1},\dots,g_n),(h_{1},\dots,h_n)])=[g_{1},\dots,g_n,h_n^{-1},h_{n-1}^{-1},\dots,h_{1}^{-1}]; \end{align*}
and so $\Psi'_n= \chi^{-1}\circ \Psi_{2n}$ is a Poisson isomorphism if so is $\Psi_{2n}$.
\end{rema} 
\begin{proof}[Proof of Theorem \ref{pro:conflaqpoi}] It is routine to check that $\Psi_{k}$ is a bijection with smooth inverse. Now we are going to check that it is a Poisson morphism. 
We may compose $\Psi_{k}$ with another Poisson isomorphism which was defined in \cite{lumou}:
\begin{align*} &J:\overbrace{G \times_{B_+} G \times_{B_+} \dots \times_{B_+} G}^{k \text{ times}} \rightarrow   \overbrace{G/B_+ \times G/B_+ \times \dots G/B_+}^{k-1 \text{ times} }\times G, \\
& J([g_1,g_2,\dots, g_{k}])=\left(g_1B_+,g_1g_2B_+,\dots,g_1g_2\dots g_{k-1}B_+,g_1g_2\dots g_k\right)
 \end{align*}   
where the codomain of $J$ is equipped with the mixed product Poisson structure of the homogeneous spaces $G/B_+$ (given by the left cosets of $B_+$) together with $G$ with respect to the residual action of $\mathfrak{g}^{k}$ on the direct product, see \cite[Thm. 8.1]{lumou}. We are going to show that $\Phi=J\circ  \Psi_{k}$ is a Poisson morphism and hence so is $\Psi_{k}$. 
 
In order to describe $\pi_{{\Sigma_k},{V_k} } $, we take a skeleton $\mathfrak{S}=(\mathfrak{S}_1,\mathfrak{S}_0)  $ for $(\Sigma_k,V_k)$. The set of vertices is $\mathfrak{S}_0=V_k$ and we take line segments from $v_{1}$ to all the other vertices in ${V_k}$ as the set of edges $\mathfrak{S}_1=\{\ee_i\}_{i=1\dots k} $. We orient these edges in such a way that their target is $v_{1}$ for all of them. The orientation of $\Sigma_k$ induces an order on ${V_k}$ and hence it also induces an order among the edges in $\mathfrak{S}_1 $. Now using the basis $\{h_i,E_{-\a},E_\a\}_{i,\a}$ of $\mathfrak{g} $ as in \S\ref{subsec:st poissson}, formula \eqref{eq:qpoimodspa} becomes:   
\[ \pi_{\Sigma_k,V_k   }=-\frac{1}{2}\sum_{a<b}\left(\sum_{i}  h_i^r(\ee_a)\wedge h_i^r(\ee_b) +\sum_{\a \in \Delta_+} (E^r_{-\a}(\ee_a) \wedge E^r_\a(\ee_b)+E^r_\a(\ee_a)\wedge E^r_{-\a}(\ee_b)) \right), \] 
where again $X(\ee_k)$ denotes the vector field whose components are $X$ on the $k$'th factor in $G^{\mathfrak{S}_1 }=G^{k}$ and $0$ on all the other factors. Fig.~\ref{fig:dis} represents the skeleton $\mathfrak{S} $ for $k=4$. On the other hand, in our situation, we have that $\pi_v=0$ for $v=v_i$, $i=2\dots k$ and $\pi_{v_1}=\Lambda\in \wedge^2 \mathfrak{g} $, $\pi_{v_{k+1}}=-\Lambda\in \wedge^2 \mathfrak{g} $. Using these observations, we may compute the terms in \eqref{eq:bivmodspa}. We have
\begin{align*} \mathtt{a}_{\Sigma_k,V_k }(\pi_{v_{1}})=&\frac{1}{2}\left(\sum_{a<b}\sum_{\a\in \Delta_+}(E^r_{-\a}(\ee_a) \wedge E^r_\a(\ee_b)-E^r_\a(\ee_a)\wedge E^r_{-\a}(\ee_b))\right.+\\
&\left.+\sum_a\sum_{\a\in\Delta_+}E^r_{-\a}(\ee_a) \wedge E^r_\a(\ee_a) \right). \end{align*} 
Then we obtain:
\begin{align*} \pi_{\Sigma_k,V_k  }+ \mathtt{a}_{\Sigma_k,V_k }(\pi_{v_{1}})=&\overbrace{-\sum_{a<b}\left(\sum_{i}  \frac{1}{2} h_i^r(\ee_a)\wedge h_i^r(\ee_b) +\sum_{\a \in \Delta_+} E^r_{\a}(\ee_a) \wedge E^r_{-\a}(\ee_b)\right)}^{-\text{Mix}(r) } +\\
 &+\frac{1}{2} \sum_a\sum_{\a\in\Delta_+}E^r_{-\a}(\ee_a) \wedge E^r_\a(\ee_a) . \end{align*}
According to \eqref{eq:bivmodspa}, the full bivector field $Q$ on the moduli space is given by
\begin{align*}  Q=\pi_{\Sigma_k,V_k   }+ \mathtt{a}_{\Sigma_k,V_k }(\pi_{v_{1}})+ \mathtt{a}_{\Sigma_k,V_k}(\pi_{v_{k+1}})=&-\text{Mix}(r) +\frac{1}{2} \left(\sum_{1\leq a \leq k-1}\sum_{\a\in\Delta_+}E^r_{-\a}(\ee_a) \wedge E^r_\a(\ee_a)\right)+ \\
&+\Lambda^r(\mathbf{e}_k) -\Lambda^l(\mathbf{e}_k) ; \end{align*} 
but $\frac{1}{2} \sum_{\a\in\Delta_+}E^r_{-\a}(\ee_a) \wedge E^r_\a(\ee_a)$ gives us the Poisson tensor on the $a$-th factor isomorphic to $G/B_+$ (see \cite[Example 6.11]{lumou}), $\Lambda^r(\mathbf{e}_k) -\Lambda^l(\mathbf{e}_k)\cong\pi_G $ and $\text{Mix}(r)$ is the mixed product term determined by $r$ as a quasi-triangular $r$-matrix, see \cite[Prop. A.6]{conpoigro}. By expressing a representation in terms of its values $x_i$ on the elements of $\mathfrak{S}_1 $, we have that $\Phi$ simplifies as follows: 
	\[ \Phi([x_1,x_2,\dots,x_{k}])=(x_{1}B_+,x_{2}B_+,\dots,x_{k-1}B_+,x_{k}). \] 
Therefore, we see that the pushforward of $Q$ by the map $\Phi$ is the opposite of the mixed product Poisson structure appearing in \cite[Prop. A.6]{conpoigro} on $(G/B_+)^{k-1}\times G$. Since our choice of $\pi_G$ differs by a sign from the one in \cite{conpoigro}, so does the Poisson tensor on $\widetilde{F}_k$, and therefore $\Phi$ is Poisson. \end{proof}
\begin{rema}\label{rem:base F2n} The Poisson groupoid structure $\widetilde{F}_{2k} \rightrightarrows F_k$ can be regarded as arising from doubling a disc with $k+1$ marked points and applying \cite[Thm. 3.5]{poigromod}: so $\Sigma_k=\Sigma_{0,k}\cup_e \Sigma_{0,k}$ where $(\Sigma_{0,k},V_{0,k})$ is the disc with $k+1$ marked points and $e$ is one of the boundary edges. Then $V$ is obtained by taking the image of $V_{0,k}\sqcup V_{0,k} \subset \Sigma$ and forgetting only one of the vertices adjacent to $e$. The decoration $(H_{0,k},\mathcal{A}_{0,k})$ of $(\Sigma_{0,k},V_{0,k})$ is then given by
\begin{align}  {\mathfrak{a} }_{v_i}=\begin{cases} &\widetilde{\mathfrak{b} }_+=\{X\oplus Y\in \mathfrak{b}_+\oplus \mathfrak{b}_+\ |\ X-Y\in \mathfrak{n}_+   \}, \quad \text{if $2\leq i\leq k+1 $}  \\
&{\mathfrak{g} }^*=\{X\oplus Y\in \mathfrak{b}_+\oplus \mathfrak{b}_-\ | \ \text{pr}_{\mathfrak{t}}(X)=- \text{pr}_{\mathfrak{t}}(Y)\}, \quad \text{if $i=1$;} \end{cases} \label{eq:dec halfsur}        \end{align} and decorating all the boundary edges with $G$. Similar considerations to those in Theorem \ref{pro:conflaqpoi} imply that the map 
\[ \mathfrak{M}_G(\Sigma_{0,k},V_{0,k})_{H_{0,k},\mathcal{A}_{0,k}} \rightarrow F_k, \qquad [a_i]_{i=1\dots k+1} \mapsto [a_1,\dots,a_{k-1},a_{k}] \]
is a Poisson isomorphism.
\end{rema} 
 
Take the following decoration $H_k^-$ in which only the edge going from $v_1$ to $v_{k+1}$ of the disc $(\Sigma_k,V_k)$ is decorated with $B_-$ and all the others are decorated with $G$. Then $\mathfrak{M}_G(\Sigma_k,V_k)_{H_{k}^-,\mathcal{A}_{k}}$ for $k=2n$ becomes a wide Poisson subgroupoid of $\widetilde{F}_{2n} \rightrightarrows F_n$ via the map $\Psi_n'$ of Remark \ref{rem:poimap}:
\begin{align} (\mathfrak{M}_G(\Sigma_{2n},V_{2n})_{H_{2n}^-,\mathcal{A}_{2n}} \rightrightarrows \mathfrak{M}_G(\Sigma_{0,2n},V_{0,2n})_{H_{0,2n},\mathcal{A}_{0,2n}} )\cong (\Gamma_{2n} \rightrightarrows F_n) \hookrightarrow (\widetilde{F}_{2n} \rightrightarrows F_n); \label{eq:confla}\end{align} 
where $\Gamma_{2n} \rightrightarrows F_n$ is the {\em $n$th total configuration Poisson groupoid of flags of $G$} introduced in \cite{conpoigro}. More generally, we denote by $\Gamma_k$ the decorated moduli space $\mathfrak{M}_G(\Sigma_k,V_k)_{H_{k}^-,\mathcal{A}_{k}}$. 

Let $W$ be the Weyl group of $(G,T)$ and consider $\mathbf{u}=(u_1,\dots,u_n)\in W^n$. The corresponding {\em generalized Schubert cell} \cite{conpoigro} is the submanifold of $F_n$
\[ \mathcal{O}^{\mathbf{u} }=B_+u_1B_+ \times_{B_+}\dots \times_{B_+}B_+u_nB_+/B_+; \]
the restriction of $\Gamma_{2n}$ to $\mathcal{O}^{\mathbf{u} }$, denoted by $\Gamma^{(\mathbf{u},\mathbf{u}^{-1})  } \rightrightarrows \mathcal{O}^{\mathbf{u} }$, is called the {\em special configuration Poisson groupoid of flags of $G$} \cite{conpoigro}. Based on Theorem \ref{pro:conflaqpoi}, these groupoids admit a moduli-theoretical interpretation. Namely, the decoration that determines $\Gamma^{(\mathbf{u},\mathbf{u}^{-1})  }$ is given as follows:
\begin{align} \Gamma^{(\mathbf{u},\mathbf{u}^{-1})  }\cong \mathfrak{M}_G(\Sigma_{2n},V_{2n})_{H'_{2n},\mathcal{A}_{2n}}  ,\qquad H_i'= \begin{cases} &B_+u_iB_+, \quad \text{if $1\leq i\leq n $}  \\
&B_+u_{2n+1-i}^{-1}B_+, \quad \text{if $n+1\leq i\leq 2n$,} \\
&B_-, \quad \text{if $i=2n+1$;} \end{cases}   \label{eq:spconfladec}      \end{align}
where $H'_{2n}=\prod_i H_i'$ and $\mathcal{A}_{2n} $ is given as before in \eqref{eq:surdec}. Here we use the bijection between vertices and boundary edges described in Remark \ref{rem:infact}. 

According to \cite[Prop. 3.15]{poigromod}, we obtain a number of symplectic groupoids by looking at the restricted groupoid structures and induced symplectic forms on the symplectic leaves passing through the unit submanifold of a Poisson groupoid that is constructed via decorated moduli spaces as in Theorem \ref{pro:conflaqpoi}. This observation contextualizes some of the results in \cite{conpoigro}, see \S\ref{subsec:symfol} for more details.

\section{Decorated moduli spaces and Morita equivalences over flag varieties}\label{sec:morequ} Based on Theorem \ref{pro:conflaqpoi}, we may integrate the Poisson structure on $\widetilde{F}_k$ by applying \cite[Thm. 3.5]{poigromod} to a suitable surface. Let $(\Sigma_k,V_k)$ be a disc with $k+1>2$ marked points $V_k=\{v_i\}_{i=1\dots k+1}$. Since $(\Sigma_k,V_k)$ is the underlying marked surface that produces $\widetilde{F}_k $, now we need to construct a new marked surface that serves as its double. Let $(\widehat{\Sigma_k},\widehat{V_k})$ be the marked surface obtained by gluing two copies of $(\Sigma_k,V_k)$ as follows. For every boundary edge $e$ in a boundary component, an arc in $e$ is chosen in such a way that it does not contain any marked points and the two copies of $\Sigma_k$ are glued along each of these arcs. Therefore, $(\widehat{\Sigma_k},\widehat{V_k})$ is a sphere with $k+1$ holes and two points on each boundary component, see Figure \ref{fig:dousur1}. For every $v\in V_k$, we denote by $\{\widehat{v},\widehat{v}'\}$ the induced pair of marked points in $(\widehat{\Sigma_k},\widehat{V_k})$ as in Figure \ref{fig:dousur1}. The groupoid structure
\begin{align}   \hom(\Pi_1(\widehat{\Sigma_k},\widehat{V_k}),G) \rightrightarrows \hom(\Pi_1(\Sigma_k,V_k),G)\label{eq:dousurdousym} \end{align} 
is described in \S\ref{app} as a particular case of a general construction. In what follows, we denote by $\widehat{E_k}$ the set of edges of the boundary graph $\widehat{\Gamma_k}$ of $(\widehat{\Sigma_k},\widehat{V_k})$. 
\begin{figure}
    \centering
   \begin{tikzpicture}[line cap=round,line join=round,>=stealth,x=1cm,y=1cm,scale=0.9]
\clip(2.8225235324454787,4.2) rectangle (9.130460334277112,8.856363326605644);
\draw [rotate around={3.2245226065199:(7.561246716552349,5.478466971264253)},line width=1.5pt] (7.561246716552349,5.478466971264253) ellipse (0.7697377303568282cm and 0.6510864772288915cm);
\fill[rotate around={3.2245226065199:(7.561246716552349,5.478466971264253)},line width=1pt,fill=black,fill opacity=0.1] (7.561246716552349,5.478466971264253) ellipse (0.7697377303568282cm and 0.6510864772288915cm);
\draw [rotate around={-2.602562202499796:(4.3612203385359996,5.409097779036143)},thick,line width=1.5pt] (4.3612203385359996,5.409097779036143) ellipse (0.8318289364893398cm and 0.6582160017758658cm);
\fill[rotate around={-2.602562202499796:(4.3612203385359996,5.409097779036143)},line width=1pt,fill=black,fill opacity=0.1] (4.3612203385359996,5.409097779036143) ellipse (0.8318289364893398cm and 0.6582160017758658cm);
\draw [rotate around={-0.20187005254749543:(5.5480200950008705,7.722016557938077)},thick,line width=1.5pt] (5.5480200950008705,7.722016557938077) ellipse (0.8215595374952916cm and 0.5037570462220995cm);
\fill[rotate around={-0.20187005254749543:(5.5480200950008705,7.722016557938077)},line width=1pt,fill=black,fill opacity=0.1] (5.5480200950008705,7.722016557938077) ellipse (0.8215595374952916cm and 0.5037570462220995cm);
\draw [shift={(2.8715863676840603,7.522299458616793)},line width=1pt]  plot[domain=-1.0361654954648998:0.09396332020550419,variable=\t]({1*1.8634558487075197*cos(\t r)+0*1.8634558487075197*sin(\t r)},{0*1.8634558487075197*cos(\t r)+1*1.8634558487075197*sin(\t r)});
\draw [shift={(8.26600990963108,7.953337640480714)},line width=1pt]  plot[domain=3.2495631571214814:4.4728897144745305,variable=\t]({1*1.9087676962775195*cos(\t r)+0*1.9087676962775195*sin(\t r)},{0*1.9087676962775195*cos(\t r)+1*1.9087676962775195*sin(\t r)});
\draw [shift={(6.054064011776575,2.5915603396976916)},line width=1pt]  plot[domain=1.1302293826980458:2.062369914000789,variable=\t]({1*2.5746147204982144*cos(\t r)+0*2.5746147204982144*sin(\t r)},{0*2.5746147204982144*cos(\t r)+1*2.5746147204982144*sin(\t r)});
\draw [shift={(3.3185760051783624,7.413923292728171)},dotted,->,line width=1pt]  plot[domain=5.501080088050842:6.194695239940627,variable=\t]({1*2.036561033253233*cos(\t r)+0*2.036561033253233*sin(\t r)},{0*2.036561033253233*cos(\t r)+1*2.036561033253233*sin(\t r)});
\draw [shift={(7.5307385649203535,8.739576574139253)},dotted,->,line width=1pt]  plot[domain=3.745408149494752:4.634210293281638,variable=\t]({1*2.651854164024706*cos(\t r)+0*2.651854164024706*sin(\t r)},{0*2.651854164024706*cos(\t r)+1*2.651854164024706*sin(\t r)});
\draw [shift={(6.5165211054636725,-4.444387086818635)},dotted,->,line width=1pt]  plot[domain=1.4943713557185048:1.737436256473725,variable=\t]({1*10.569405039806854*cos(\t r)+0*10.569405039806854*sin(\t r)},{0*10.569405039806854*cos(\t r)+1*10.569405039806854*sin(\t r)});
\begin{scriptsize}
\draw [fill=black] (5.522384437828968,8.22558660947952) circle (1.5pt);
\draw[color=black] (5.51410459063444,8.5) node {$\widehat{v}_1$};
\draw [fill=black] (4.7633763633297646,5.9786070323691325) circle (1.5pt);
\draw[color=black] (4.6,5.7) node {$\widehat{v}_2'$};
\draw [fill=black] (5.347797926123081,7.233887139115374) circle (1.5pt);
\draw[color=black] (5.275610066491114,7.5) node {$\widehat{v}_1'$};
\draw[color=black] (5.085787894213772,6.7) node {$g_1$};
\draw [fill=black] (7.323501478998941,6.094166209161958) circle (1.5pt);
\draw[color=black] (7.4,5.8) node {$\widehat{v}_3'$};
\draw[color=black] (6.107907283399454,6.7) node {$g_3$};
\draw[color=black] (5.747731879591166,6.25) node {$g_2$};
\draw [fill=black] (4.373562714203231,4.750491006354823) circle (1.5pt);
\draw[color=black] (4.355702616224002,4.5) node {$\widehat{v}_2$};
\draw [fill=black] (7.524738680638493,4.8272370621510365) circle (1.5pt);
\draw[color=black] (7.558343369005803,4.5) node {$\widehat{v}_3$};
\end{scriptsize}
\end{tikzpicture}\begin{tikzpicture}[line cap=round,line join=round,>=stealth,x=1cm,y=1cm,scale=0.3]
\clip(-0.29689350774554574,0.8) rectangle (9.838699598877637,11.7);
\draw [rotate around={2.4681179138281935:(4.897030347083754,7.969806602197145)},line width=1pt] (4.897030347083754,7.969806602197145) ellipse (2.5463628795096085cm and 2.4962738173387544cm);
\draw [dotted,line width=1pt,->] (4.7541636680191385,10.465)-- (4.8541636680191385,10.465744956679437);
\draw [dotted,line width=1pt,->] (4.9653017228604085,5.473583249803503)-- (4.8653017228604085,5.473583249803503);
\begin{scriptsize}
\draw [fill=black] (4.8541636680191385,10.465744956679437) circle (1.5pt);
\draw[color=black] (4.865265084346218,11.2) node {$\widehat{v}'$};
\draw[color=black] (2.865265084346218,8.1) node {$a$};
\draw[color=black] (6.865265084346218,8.1) node {$b$};
\draw [fill=black] (4.8653017228604085,5.473583249803503) circle (1.5pt);
\draw[color=black] (5.020684912925324,6.1) node {$\widehat{v}$};
\end{scriptsize}
\end{tikzpicture} 
    \caption{Left: The double surface $(\widehat{\Sigma_2},\widehat{V_2})$ corresponding to the disc with three marked points $(\Sigma_2,V_2)$ determines an integration of $\widetilde{F}_2$. Right: A boundary component decorated with $(a,b^{-1})\in L_v$}
    \label{fig:dousur1}
\end{figure}  
Consider the following decoration of $(\widehat{\Sigma_k},\widehat{V_k})$. The vertices $\{\widehat{v}_i,\widehat{v}_i'\}$ are decorated by:
\[ \mathfrak{l}_{\widehat{v}_i}=\mathfrak{l}_{v_i}, \qquad \mathfrak{l}_{\widehat{v}_i'}=\mathfrak{l}_{v_i}^\vee;  \]
where $\mathfrak{l}_{v_i}^\vee=\{(Y\oplus X)\ |\ X\oplus Y\in \mathfrak{l}_{v_i}\}$ and $\mathfrak{l}_{v_i}$ is as in \eqref{eq:surdec}, determining the Lie algebra $\widehat{\mathcal{A}_k} =\bigoplus_{i=1}^{k+1}(\mathfrak{l}_{\widehat{v}_i} \oplus \mathfrak{l}_{\widehat{v}_i'})\subset \mathfrak{d}^{2k+2}  $. Let $\mathcal{L}_k \subset G^{\widehat{E} }$ be the orbit of $\widehat{\mathcal{A}_k} $ that passes through the unit $1\in G^{\widehat{E} }$. To simplify notation, we will denote by 
\begin{align}  \mathcal{G}_{k} \rightrightarrows \widetilde{F}_{k} \label{eq:sym gro dou sur} \end{align}  
the symplectic groupoid structure on $\mathfrak{M}_G(\widehat{\Sigma_k},\widehat{ V_k})_{\mathcal{L}_k,\widehat{\mathcal{A}_k}   }$ constructed in the next proposition.

\begin{prop}\label{pro:intcon} We have that $\mathcal{G}_k  \rightrightarrows \widetilde{F}_{k}$ is a complex algebraic symplectic groupoid that integrates the Poisson structure on $\widetilde{F}_{k}$ introduced in \S \ref{subsec:poigrofla}. 
 \end{prop}  
\begin{rema}\label{rem:dec dou sur} The orbit $\mathcal{L}_k $ can be described explicitly as follows. We label the edges of a boundary component as in Fig.~\ref{fig:dousur1}. Then the decoration demands that $(a,b^{-1})$ lie in suitable subgroups of $G \times G$. 
Take the lagrangian subgroup $L \subset (G \times {G})^V =(G \times G)^{k+1} $ as follows: $L=\prod_{i=1}^{k+1} L_{v_i}$, where $L_{v_i}$ is the integration of $\mathfrak{l}_{v_i}$ as in \eqref{eq:surdec} defined by
\begin{align*}  L_{v_i}=\begin{cases} &\widetilde{B }_+=\{(a, b)\in B_+\times  B_+\ | \ ab^{-1}\in N_+   \}, \quad \text{if $2\leq i\leq k $}  \\
&G^*=\{(a, b)\in B_+ \times B_-\ |\ \text{pr}_T(a) \text{pr}_{T}(b)=1\}, \quad \text{if $i=1$,} \\
&(G^*)^\vee=\{(a,b)\in G^2\ |\ (b,a)\in G^*\}, \quad \text{if $i=k+1$;} \end{cases}    \end{align*}
where $\text{pr}_{T}:B_\pm \rightarrow T$ are the quotient maps. So we have that the map $L \rightarrow G^{\widehat{E_k} }=(G \times G)^{k+1}$ given by $(a_i,b_i)_i\mapsto (a_i,b_i^{-1})_i$ embeds $L$ as $\mathcal{L}_k $.
\end{rema} 
\begin{proof}[Proof of Proposition \ref{pro:intcon}] 
We can see that the moment map $\mu:\hom(\Pi_1(\widehat{\Sigma_k},\widehat{V_k}),G) \rightarrow( G \times G)^{k+1}$ is  a Lie groupoid morphism, where the target is regarded as a product of Lie groups and on each factor $G \times G$, the second copy is equipped with the opposite group structure, see \cite[Prop. 3.11]{poigromod}. 

The smoothness of this moduli space follows from the definitions introduced in the appendix. Let $\mathcal{E}$ be the Lie groupoid introduced at the beginning of \S\ref{subsec:expgro} with moment map $\mu$ \eqref{eq:mommapgro}. Let $\mathcal{E}'= \mathcal{E}\times_{G \times G} \widetilde{B}_+$ be the fibred product with respect to the last two components of $\mu$. This corresponds to decorating the outer boundary of the annulus in Fig.~\ref{fig:d(g)} with $\widetilde{B}_+$. Take the quotient of $\mathcal{E}'$ corresponding to the gauge action of $B_+^2$ at the two marked points on the decorated boundary. Such a quotient is isomorphic to $\mathcal{G}'= G \times (G \times N_+)/B_+$, where $(G \times N_+)/B_+$ is the quotient of $G \times N_+$ by the $B_+$-action $b\cdot (g,u)=(gb^{-1},bub^{-1})$. The isomorphism is as follows
\[ [a_1,a_2,g_1,g_2]\mapsto (g_1,[a_1,g_2a_2^{-1}g_1^{-1}a_1 ]). \] 
We have a residual moment map $\mu':\mathcal{G}'\rightarrow G \times G$ defined by $(a,[g,u])\mapsto (a,gug^{-1}a)$. Note that $\mathcal{G}'$ may be viewed as the moduli space of an annulus with unipotent monodromy at the outer boundary and two marked points on the other. Taking the fibred product with respect to the components of the moment map $\mu'$, which are pairwise transverse over $G$, leads to
\[ \mathcal{E}_{k-1}= \overbrace{\mathcal{E}'\times_G \mathcal{E}'\dots \times_G\mathcal{E}'}^{\text{$k-1$ times} }; \] 
that corresponds to gluing $k-1$ such annuli along their boundary arcs to create a sphere with $k$ holes, one with two marked points and no constraints and the others with unipotent monodromy around them. Finally, take the fibred product $\mathcal{M}_k= \mathcal{E}_{k-1} \times_G \mathcal{E}$ which corresponds to adding another hole with two marked points to such a sphere. Note that the moment map of $\mathcal{M}_k$ has target $(G \times G)^2$ and it restricts to a surjective submersion onto the respective two diagonals in the two copies of $G \times G$. So decorating these free two boundaries with $G^*$ and $(G^* )^\vee$, which are Lie subgroups transverse to the diagonals, gives rise to a smooth fibred product, which is the desired $\mathfrak{M}_G(\widehat{\Sigma_k},\widehat{ V_k})_{\mathcal{L}_k,\widehat{\mathcal{A}_k}   }$.   

It follows that $\mathfrak{M}_G(\widehat{\Sigma_k},\widehat{ V_k})_{\mathcal{L}_k,\widehat{\mathcal{A}_k}   }$ is a symplectic groupoid that integrates $\widetilde{F}_k$, according to \cite[Prop. 3.15, Prop. 3.8]{poigromod}. The fact that it is a complex algebraic variety with algebraic structure maps follows from Proposition \ref{pro:slim}; its symplectic form is obtained from \eqref{eq:polwie} and hence it is also algebraic, see \S\ref{subsec:mom map surf}. \end{proof} 

\subsection{Morita equivalences of symplectic double groupoids}\label{sec:dousym} 
Consider a groupoid object in the category of Lie groupoids. Such a structure is denoted by a diagram of the following kind
\begin{align}   \xymatrix{ \mathcal{G}  \ar@<-.5ex>[r] \ar@<.5ex>[r]\ar@<-.5ex>[d] \ar@<.5ex>[d]& \mathcal{H}   \ar@<-.5ex>[d] \ar@<.5ex>[d] \\ \mathcal{K}   \ar@<-.5ex>[r] \ar@<.5ex>[r] & M, } \label{eq:dousym} \end{align}   
where each of the sides represents a groupoid structure and the structure maps of $\mathcal{G}  $ over $\mathcal{H} $ are groupoid morphisms with respect to $\mathcal{G}  \rightrightarrows \mathcal{K} $ and  $\mathcal{H}  \rightrightarrows M$. A {\em double Lie groupoid} is an object as above such that the map $(\mathtt{s}^V,\mathtt{s}^H):  \mathcal{G}  \rightarrow \mathcal{K}  \times_{M} \mathcal{H}  $ is a submersion. The superindices $\quad^H,\quad^V$ denote, respectively, the groupoid structures $\mathcal{G}  \rightrightarrows \mathcal{H} $ and $\mathcal{G}  \rightrightarrows \mathcal{K} $ \cite{browmac,macdou} and $\mathcal{K} \rightrightarrows M$, $\mathcal{H} \rightrightarrows M$ are called the {\em side groupoids} of $\mathcal{G} $.
\begin{defi}[\cite{weicoi}] A symplectic double groupoid is a double Lie groupoid as in \eqref{eq:dousym} equipped with a symplectic form on $\mathcal{G} $ which is multiplicative, with respect to both vertical and horizontal groupoid structures. \end{defi}
\begin{exa}[Double pair groupoid] One of the simplest examples of double (symplectic) groupoids is the {\em double pair groupoid} associated to a manifold $M$. We consider $M^4$ equipped with the two multiplications 
\begin{align*} &\mathtt{m}^{H}((a,b,c,d),(c,d,a',b'))= (a,b,a',b') \\
&\mathtt{m}^{V}((a,b,c,d),(b,x,d,y))= (a,x,c,y) \end{align*}
for all $a,b,c,d,a',b',x,y\in M$. They define the horizontal and vertical groupoid structures $M^4 \rightrightarrows  M^2$, so both side groupoids are given by the pair groupoid over $M$. If $(M,\omega )$ is symplectic, the symplectic form $\omega \times (-\omega ) \times (-\omega ) \times \omega  $ makes $M^4$ into a symplectic double groupoid. \end{exa} 
 
In order to define the concept of symplectic Morita equivalences we need to recall groupoid actions. A Lie groupoid $G \rightrightarrows M$ acts on a map $j:S \rightarrow M$ on the left (respectively, on the right) if there is a map $G_{\mathtt{s} }\times_j S \rightarrow S$ (resp. $S_j \times_{\mathtt{t} } G \rightarrow S$) satisfying $\mathtt{u}(j(s))\cdot s=s$ and $g\cdot (h\cdot s)=(gh)\cdot s$ for all suitable $g,h\in G$ and $s\in S$; here we are denoting by $(g,s)\mapsto g\cdot s$ such an action map. The conditions for a right action are analogous. 

Take a Lie groupoid $G \rightrightarrows M $ acting on the left (right) on a map $j:S \rightarrow M$. The associated \emph{action groupoid} $G \ltimes S \rightrightarrows S$ (respectively, $S \rtimes G \rightrightarrows S$) has $G \ltimes S:=G \times_M S=G_\mathtt{s}\times_j S$ (respectively, $S \rtimes G:=S_j\times_\mathtt{t} G$) as the space of arrows and $S$ as the space of objects. The source is the projection to $S$ and the target is the action map $G_\mathtt{s}\times_j S\rightarrow S$. The multiplication is $\mathtt{m}( (a,b\cdot x),(b,x))=(ab,x)$ (respectively, $\mathtt{m}( (x,a),(x\cdot a,b))=(x,ab)$), the unit map $S\rightarrow G\ltimes S$ is the inclusion $x\mapsto (\mathtt{u}(j(x)),x)$ and the inversion map is $(a,x)\mapsto (a^{-1},a\cdot x)$. If $G \rightrightarrows M $ is a symplectic groupoid and $S$ is a symplectic manifold, such an action is {\em symplectic} if its graph is lagrangian $\{(g,s,g\cdot s)\}\subset G \times S \times \overline{S}$, where we denote by $\overline{S}$ the space $S$ but equipped with the opposite symplectic structure. 
\begin{defi}[\cite{morsym}] A {\em symplectic Morita equivalence} between two symplectic groupoids $G \rightrightarrows M$ and $H \rightrightarrows N$ consists of a left symplectic groupoid action $G \ltimes S \rightarrow S$ and a right symplectic groupoid action $S \rtimes H \rightarrow S$ on surjective submersions $S \rightarrow M$ and $S \rightarrow N$ with the properties that: (1) both actions commute and (2) they induce two identifications of orbit spaces $S/G\cong N$ and $S/H\cong M$. In this situation, $S$ is called a {\em symplectic Morita bimodule}. \end{defi}
There is a natural extension of this concept for symplectic double groupoids. However, Morita equivalences for double Lie groupoids may be of two types: either vertical or horizontal, see \cite{intgk}. Here we shall only need one version of the general definition.  
\begin{defi}[\cite{intgk}] A horizontal Morita equivalence of symplectic double Lie groupoids is a diagram as below
\[ \begin{tikzcd}
\mathcal{G} \arrow[r, shift right] \arrow[r, shift left] \arrow[d, shift right] \arrow[d, shift left] & \mathcal{K} \arrow[d, shift right] \arrow[d, shift left] & \mathcal{Z} \arrow[r, "q"] \arrow[l, "p"'] \arrow[d, shift right] \arrow[d, shift left] & \mathcal{K}' \arrow[d, shift right] \arrow[d, shift left] & \mathcal{G}' \arrow[l, shift left] \arrow[l, shift right] \arrow[d, shift right] \arrow[d, shift left] \\
\mathcal{H} \arrow[r, shift right] \arrow[r, shift left]                                              & M                                                        & Z \arrow[l, "p_0"'] \arrow[r, "q_0"]                                                    & M'                                                        & \mathcal{H}' \arrow[l, shift left] \arrow[l, shift right]                                             
\end{tikzcd}\]
in which the top row represents a usual symplectic Morita equivalence of symplectic groupoids that relates the horizontal symplectic groupoid structures and, moreover, all the structure maps are groupoid morphisms with respect to the vertical groupoid structure. This means that $\mathcal{Z} $ is a symplectic groupoid over $Z$, that $p,q$ are groupoid morphisms over $p_0,q_0$, respectively, and the action maps $\mathcal{G} \times_{\mathcal{K} } \mathcal{Z} \rightarrow \mathcal{Z}$ and $\mathcal{Z} \times_{\mathcal{K}'} \mathcal{G}'\rightarrow \mathcal{Z}$ are groupoid morphisms as well.

To illustrate graphically horizontal double groupoid actions, we can represent the corresponding elements as squares and the action as concatenation. Also, vertical multiplication of elements in a double groupoid can be represented as vertical concatenation of squares, see Fig.~\ref{fig:resmor} below.
\end{defi}  
\subsection{A symplectic Morita equivalence between $\mathcal{G}_{2n} $ and $\mathcal{G}_{2m}$} One of the most powerful ways of producing such Morita equivalences is given by cutting and gluing decorated surfaces as in \cite{sevmorqua,intgk}. We discuss this general construction in \S\ref{subsec:mor}.

Now we will show that the symplectic groupoids \eqref{eq:sym gro dou sur} are actually double groupoids and they integrate the Poisson groupoid structure $\widetilde{F}_{2n} \rightrightarrows F_n$. We start by introducing a Poisson groupoid over $F_n$ that is in duality with respect to $\widetilde{F}_{2n} $, in the sense that they are the side groupoids of a common symplectic double groupoid that integrates them. Consider the disc with $n+1$ marked points as our marked surface $(\Sigma_{0,n},V_{0,n}) $ and decorated as in Remark \ref{rem:base F2n}. Now take its double surface just as before: by gluing two copies of $\Sigma_{0,n}$ along an arc between each pair of consecutive marked points to produce $(\widehat{\Sigma_{0,n}},\widehat{V_{0,n}})$, a sphere with $n+1$ holes and two marked points on each boundary component. The decoration $(\widehat{H_{0,n}},\widehat{\mathcal{A}_{0,n} })  $of $(\widehat{\Sigma_{0,n}},\widehat{V_{0,n}})$ is as follows:  
\[ \mathfrak{a}_{\widehat{v}_i}=\mathfrak{a}_{v_i}, \qquad \mathfrak{a}_{\widehat{v}_i'}=\mathfrak{a}_{v_i}^\vee;  \]   
where $\mathfrak{a}_{v_i} $ is as in \eqref{eq:dec halfsur}; moreover the boundary components are given the following decoration: we demand that the values $(a,b^{-1})$ of a representation on a boundary component as in Remark \ref{rem:dec dou sur} lie in 
\begin{align}  L_{v_i}=\begin{cases} &\widetilde{B }_+=\{(a, b)\in B_+\times  B_+\ | \ ab^{-1}\in N_+   \}, \quad \text{if $2\leq i\leq n $}  \\
&B_+ \times G, \quad \text{if $i=n+1$,} \\
&G^*=\{(a, b)\in B_+ \times B_-\ |\ \text{pr}_T(a) \text{pr}_{T}(b)=1\}, \quad \text{if $i=1$.} \end{cases}  \label{eq:decduagro}  \end{align}
As a consequence of \cite[Thm. 3.5]{poigromod}, there is then a Poisson groupoid structure on the decorated moduli space corresponding to the decoration we have just described that we denote as $\widetilde{F}_{2n}^* \rightrightarrows F_n$. Such a Poisson groupoid structure arises from reduction of a groupoid structure on 
 \begin{align}   \hom(\Pi_1(\widehat{\Sigma_{0,n}},\widehat{V_{0,n}}),G) \rightrightarrows \hom(\Pi_1(\Sigma_{0,n},V_{0,n}),G), \label{eq:dousurdual} \end{align}  
just as described in \S\ref{subsec:poigromod}, see Prop.~\ref{pro:expduapoi} for explicit formulae.
\begin{prop}\label{pro:symdou} There is a complex algebraic symplectic double groupoid structure on $\mathcal{G}_{2n} \rightrightarrows \widetilde{F}_{2n}  $ as in \eqref{eq:sym gro dou sur} over $\widetilde{F}_{2n}^* \rightrightarrows F_n$. \end{prop}    
\begin{proof} We just have to observe that the Poisson groupoid structures $\widetilde{F}_{2n} \rightrightarrows F_n$, $\widetilde{F}_{2n}^* \rightrightarrows F_n$ and $\mathcal{G}_{2n} \rightrightarrows \widetilde{F}_{2n}  $ arise via doubling surfaces as recalled in \S\ref{subsec:poigromod}, following \cite[Thm. 3.5]{poigromod}. Now, $(\widehat{\Sigma_{2n}},\widehat{V_{2n}})$ as in Proposition \ref{pro:intcon} may also be viewed as the double of the marked surface $(\widehat{\Sigma_{0,n}},\widehat{V_{0,n}})$ by cutting $\widehat{\Sigma_{2n}}$ along a path that joins the two marked points $\widehat{v}_{2n+1}$ and $\widehat{v}_{2n+1}'$, see Fig. \ref{fig:dual}. Since the decoration of $(\widehat{\Sigma_{2n}},\widehat{V_{2n}})$ is symmetric with respect to either of these two decompositions and the involutions determined by these commute (they correspond to the vertical and horizontal reflection symmetries of Fig.~\ref{fig:triang}), we have that \cite[Cor. 4.9]{poigromod} implies that there is indeed a symplectic double groupoid structure on $\mathcal{G}_{2n}$ with sides $\widetilde{F}_{2n}$ and $\widetilde{F}_{2n}^*$.   \end{proof}
\begin{figure}
    \centering
    \begin{tikzpicture}[line cap=round,line join=round,>=triangle 45,x=1cm,y=1cm,scale=0.9]
\clip(2.506181003781667,4.128678009836118) rectangle (8.671945101050948,8.81960655914752);
\draw [rotate around={3.2245226065199:(7.561246716552349,5.478466971264253)},line width=1.5pt] (7.561246716552349,5.478466971264253) ellipse (0.7697377303568282cm and 0.6510864772288915cm);
\fill[rotate around={3.2245226065199:(7.561246716552349,5.478466971264253)},line width=1.5pt,fill=black,fill opacity=0.1] (7.561246716552349,5.478466971264253) ellipse (0.7697377303568282cm and 0.6510864772288915cm);
\draw [rotate around={-2.602562202499796:(4.3612203385359996,5.409097779036143)},line width=1.5pt] (4.3612203385359996,5.409097779036143) ellipse (0.8318289364893398cm and 0.6582160017758658cm);
\fill[rotate around={-2.602562202499796:(4.3612203385359996,5.409097779036143)},line width=1.5pt,fill=black,fill opacity=0.1] (4.3612203385359996,5.409097779036143) ellipse (0.8318289364893398cm and 0.6582160017758658cm);
\draw [rotate around={-0.20187005254749543:(5.5480200950008705,7.722016557938077)},line width=1.5pt] (5.5480200950008705,7.722016557938077) ellipse (0.8215595374952916cm and 0.5037570462220995cm);
\fill[rotate around={-0.20187005254749543:(5.5480200950008705,7.722016557938077)},line width=1.5pt,fill=black,fill opacity=0.1] (5.5480200950008705,7.722016557938077) ellipse (0.8215595374952916cm and 0.5037570462220995cm);
\draw [shift={(2.8715863676840603,7.522299458616793)},line width=1pt]  plot[domain=-1.0361654954648998:0.09396332020550419,variable=\t]({1*1.8634558487075197*cos(\t r)+0*1.8634558487075197*sin(\t r)},{0*1.8634558487075197*cos(\t r)+1*1.8634558487075197*sin(\t r)});
\draw [shift={(8.26600990963108,7.953337640480714)},line width=1pt]  plot[domain=3.2495631571214814:4.4728897144745305,variable=\t]({1*1.9087676962775195*cos(\t r)+0*1.9087676962775195*sin(\t r)},{0*1.9087676962775195*cos(\t r)+1*1.9087676962775195*sin(\t r)});
\draw [shift={(6.054064011776575,2.5915603396976916)},line width=1pt]  plot[domain=1.1302293826980458:2.062369914000789,variable=\t]({1*2.5746147204982144*cos(\t r)+0*2.5746147204982144*sin(\t r)},{0*2.5746147204982144*cos(\t r)+1*2.5746147204982144*sin(\t r)});
\draw [shift={(3.71576488246503,6.262864281793167)},dashed,line width=1pt]  plot[domain=-0.47493556323007535:0.4299875279359948,variable=\t]({1*2.413890136814543*cos(\t r)+0*2.413890136814543*sin(\t r)},{0*2.413890136814543*cos(\t r)+1*2.413890136814543*sin(\t r)});
\draw [shift={(8.037944529685683,6.726488032851255)},dashed,line width=1pt]  plot[domain=2.572912158305599:3.748606066885096,variable=\t]({1*2.765616596417033*cos(\t r)+0*2.765616596417033*sin(\t r)},{0*2.765616596417033*cos(\t r)+1*2.765616596417033*sin(\t r)});
\begin{scriptsize}
\draw [fill=black] (4.7633763633297646,5.9786070323691325) circle (1.5pt);
\draw[color=black] (4.723191859744232,6.2339786016563576) node {$\widehat{v}_1'$};
\draw [fill=black] (7.32350147899894,6.094166209161958) circle (1.5pt);
\draw[color=black] (7.0,6.229129282169773) node {$\widehat{v}_2'$};
\draw [fill=black] (6.79386120054327,5.419191211542009) circle (1.5pt);
\draw[color=black] (7.1,5.467923838062447) node {$\widehat{v}_2$};
\draw [fill=black] (5.192313481390507,5.362662267018276) circle (1.5pt);
\draw[color=black] (4.9,5.3537430214463475) node {$\widehat{v}_1$};
\draw [fill=black] (5.910617475132691,7.269436692355691) circle (1.5pt);
\draw[color=black] (5.964908240444297,7.6) node {$\widehat{v}_3'$};
\draw [fill=black] (5.707601998674813,8.215829782262446) circle (1.5pt);
\draw[color=black] (5.760334277340455,8.55) node {$\widehat{v}_3$};
\end{scriptsize}
\end{tikzpicture}  
    \caption{The double surface of a disc with three marked points; by cutting it along the dashed arc we obtain the marked surface $\widehat{\Sigma_{0,1}} $ corresponding to $\widetilde{F}^*_2$}
    \label{fig:dual}
\end{figure}
In the proof of the following theorem, we will review how such a double groupoid structure arises. An alternative description of such a double groupoid structure in terms of its side groupoids is given in Proposition \ref{pro:slim}.

\begin{rema} In what follows, we view the groupoid structure of $\mathcal{G}_{2n}$ over $\widetilde{F}_{2n}^*$ as horizontal and the one over $\widetilde{F}_{2n}$ as vertical. \end{rema}
   
\begin{thm}\label{thm:symdoumor} The symplectic groupoid $\mathcal{G}_{n+m} \rightrightarrows \widetilde{F}_{n+m}$ serves as the symplectic Morita bimodule for a horizontal symplectic Morita equivalence between the symplectic double groupoids $\mathcal{G}_{2n} $ and $\mathcal{G}_{2m}$:
\[ \begin{tikzcd}
\mathcal{G}_{2m} \arrow[r, shift right] \arrow[r, shift left] \arrow[d, shift right] \arrow[d, shift left] & \widetilde{F}^*_{2m}  \arrow[d, shift right] \arrow[d, shift left] & \mathcal{G}_{n+m} \arrow[r, "q"] \arrow[l, "p"'] \arrow[d, shift right] \arrow[d, shift left] & \widetilde{F}_{2n}^*  \arrow[d, shift right] \arrow[d, shift left] & \mathcal{G}_{2n} \arrow[l, shift left] \arrow[l, shift right] \arrow[d, shift right] \arrow[d, shift left] \\
\widetilde{F}_{2m}  \arrow[r, shift right] \arrow[r, shift left]                                              & F_m                                                        & \widetilde{F}_{n+m}   \arrow[l, "p_0"'] \arrow[r, "q_0"]                                                    & F_n                                                        & \widetilde{F}_{2n}; \arrow[l, shift left] \arrow[l, shift right]                                             
\end{tikzcd}\]
moreover, the action maps are morphisms of complex algebraic varieties.
   \end{thm}
\begin{proof} The strategy is to apply Theorem \ref{thm:morequ} to produce the Morita equivalence in the first row in the diagram above. Since such a Morita equivalence is obtained by reduction from a Morita equivalence of pair groupoids, we just need to show how it is actually induced by a horizontal Morita equivalence of double pair groupoids to show its compatibility with the vertical groupoid structures.  

{\em Step 1: the double groupoid structure on $\mathcal{G}_{2n} $}. Denote $\hom(\Pi_1(\Sigma_{0,k},V_{0,k}),G)$ by $P_k$, we use the conventions of Remark \ref{rem:infact} in what follows. Following the construction of \cite[Prop. 3.11,Thm. 3.5]{poigromod} as reviewed in \S\ref{subsec:mor}, the groupoid structure of \eqref{eq:dousurdual} is obtained as the quotient 
\[ G^{2n+2}\backslash (P_{3n+2} \times_{S_n} P_{3n+2}) \rightrightarrows P_n, \qquad S_n=\{e_3,e_6,e_9,\dots \}; \] 
where the fibred product is the one defined as in \eqref{eq:quomodspa} and the action is \eqref{eq:gauact} restricted to the copies of $G$ associated with the vertices adjacent to edges in $S_n$. The double groupoid of Theorem \ref{pro:symdou} is obtained as in the proof of \cite[Thm. 4.7]{poigromod}, namely, we start with the disjoint union of four copies of $\Sigma_{0,3n+2}$ and then we glue them and delete marked points until we obtain $(\widehat{\Sigma_{2n}},\widehat{V_{2n}} )$. In terms of representation varieties, we take the double pair groupoid over $P_{3n+2}$ and then its double subgroupoid corresponding to the gluing pattern to be described below:
\[ \begin{tikzcd}[scale cd=.85]
{\mathcal{M}_{2n}:= (P_{3n+2}\times_{S_n} P_{3n+2})\times_{\{\iota_1(e_2)\ast\iota_2(e_2)\}}(P_{3n+2}\times_{S_n} P_{3n+2})} \arrow[d, shift right] \arrow[d, shift left] \arrow[rr, shift right] \arrow[rr, shift left] &  & P_{3n+2}\times_{S_n} P_{3n+2} \arrow[d, shift right] \arrow[d, shift left] \\
P_{3n+2}\times_{\{e_2\}}P_{3n+2} \arrow[rr, shift right] \arrow[rr, shift left]                                                                                                           &  & P_{3n+2};                                                              
\end{tikzcd} \]
so $\mathcal{M}_{2n}$ corresponds to the representations of the fundamental groupoid of the surface obtained by gluing four copies of $\Sigma_{0,3n+2}$ as follows. First we glue the copies of $\Sigma_{0,3n+2}$ pairwise along the two inclusions of $S_n$, this creates two inclusions $\iota_i:\Sigma_{0,3n+2}\hookrightarrow  \Sigma_{0,3n+2}\cup_{S_n}\Sigma_{0,3n+2}$ for $i=1,2$. Then in this new surface, $\iota_1(e_2)$ and $\iota_2(e_2)$ are glued to create a single boundary edge $\iota_1(e_2)\ast\iota_2(e_2)$ and we glue the two copies of $ \Sigma_{0,3n+2}\cup_{S_n}\Sigma_{0,3n+2}$ alongside it:
 \begin{align}  \widehat{ \Sigma_{2n}}\cong (\Sigma_{0,3n+2}\cup_{S_n}\Sigma_{0,3n+2})\cup_{\{\iota_1(e_2)\ast\iota_2(e_2)\}} (\Sigma_{0,3n+2}\cup_{S_n}\Sigma_{0,3n+2}); \label{eq:glu4} \end{align}  
where the double surface $\widehat{ \Sigma_{2n}}$ is the one used in \eqref{eq:dousurdousym}. To obtain the corresponding groupoid structure in \eqref{eq:dousurdousym}, we are required to apply the gauge action \eqref{eq:gauact} restricted to a suitable subgroup of $G^{4V_{0,3n+2}}$ which is in fact a double groupoid itself:
\[ \begin{tikzcd}[scale cd=.85]
\mathcal{K}_{2n}:= G^{2n+1}\times G^{2n+1}\times G_\Delta \arrow[d, shift right] \arrow[d, shift left] \arrow[rr, shift right] \arrow[rr, shift left] &  &  G^{2n+2} \arrow[d, shift right] \arrow[d, shift left] \\
G^{2n+1}\times G^{2n+1}\times G_\Delta \arrow[rr, shift right] \arrow[rr, shift left]                                              &  & G^{2n+2};                                             
\end{tikzcd} \]
horizontally, this is the product of the pair groupoid over $G^{2n+1}$ and the unit groupoid over $G$, vertically, it is a unit groupoid. As a subgroup of $G^{4V_{0,3n+2}}$, we have that $G^{2n+1}\times G^{2n+1}\times G_\Delta$ corresponds to the vertices adjacent to all the edges in $S_n$, with the diagonal factor being $G_\Delta \subset G_{\mathtt{T}(e_2)}^2$ (the subindex denoting the corresponding vertex).

The upshot of this discussion is that we obtain a double groupoid as the quotient $\mathcal{M}_{2n}/\mathcal{K}_{2n}$:  \[ \begin{tikzcd}[scale cd=.85]
{\hom(\Pi_1(\widehat{\Sigma_{2n}},\widehat{V_{2n}}),G)} \arrow[d, shift right] \arrow[d, shift left] \arrow[rr, shift right] \arrow[rr, shift left] &  & {\hom(\Pi_1(\widehat{\Sigma_{0,n}},\widehat{V_{0,n}}),G)} \arrow[d, shift right] \arrow[d, shift left] \\
{\hom(\Pi_1(\Sigma_{2n},V_{2n}),G)} \arrow[rr, shift right] \arrow[rr, shift left]                                                                  &  & {\hom(\Pi_1(\Sigma_{0,n},V_{0,n}),G);}                                                                 
\end{tikzcd}\] 
the symplectic double groupoid structure on $\mathcal{G}_{2n} $ arises from reducing the above double groupoid using the decoration of Remark \ref{rem:dec dou sur}.

{\em Step 2: the symplectic bimodule}. Now we can obtain $\mathcal{G}_{n+m}$ as the reduction of the representation variety of the disjoint union of two copies of $\Sigma_{0,3n+2}$ and two copies of $\Sigma_{0,3m+2}$. This means that we should consider the representation variety $\mathcal{M}_{n+m}:= (P_{3m+2}\times_{S_m} P_{3m+2})\times_{\{\iota_1(e_2)\ast\iota_2(e_2)\}}(P_{3n+2}\times_{S_n} P_{3n+2})$ which naturally functions as a bimodule for a Morita equivalence of double groupoids between $\mathcal{M}_{2n}$ and $\mathcal{M}_{2m}$. The diagram below illustrates only one of the projections for the two actions. 
\[ \begin{tikzcd}[scale cd=.85]
\mathcal{M}_{n+m} \arrow[d, shift right] \arrow[d, shift left] \arrow[r] & P_{3n+2}\times_{S_n} P_{3n+2} \arrow[d, shift right] \arrow[d, shift left] & \mathcal{M}_{2n} \arrow[l, shift left] \arrow[l, shift right] \arrow[d, shift right] \arrow[d, shift left] \\
P_{3m+2}\times_{\{e_2\}}P_{3n+2} \arrow[r]                               & P_{3n+2}                                                               & P_{3n+2}\times_{\{e_2\}}P_{3n+2}. \arrow[l, shift left] \arrow[l, shift right]                             
\end{tikzcd}\]
The horizontal actions of $\mathcal{M}_{2n}$ and $\mathcal{M}_{2m}$ on $\mathcal{M}_{n+m}$ are compatible with the vertical groupoid structure and hence they descend to a horizontal Morita equivalence of double groupoids between $\hom(\Pi_1(\widehat{\Sigma_{2n}},\widehat{V_{2n}}),G)$ and $\hom(\Pi_1(\widehat{\Sigma_{2m}},\widehat{V_{2m}}),G)$, as described in Theorem \ref{thm:morequ}. Notice that, after reducing by the gauge action corresponding to the vertices adjacent to the edges that are being glued, $\mathcal{M}_{n+m} \rightrightarrows 
 P_{3m+2}\times_{\{e_2\}}P_{3n+2}$ determines the Lie groupoid $\hom(\Pi_1(\widehat{\Sigma_{n+m}},\widehat{V_{n+m}}),G) \rightrightarrows \hom(\Pi_1({\Sigma_{n+m}},{V_{n+m}}),G)$. 

After reducing the spaces $\mathcal{M}_{2n}$, $\mathcal{M}_{n+m}$ and $\mathcal{M}_{2m}$ using their corresponding decorations as in Remark \ref{rem:dec dou sur}, we get a pair of commuting symplectic actions of $\mathcal{G}_{2n}$ and $\mathcal{G}_{2m}$ on $\mathcal{G}_{n+m}$. The fact that the actions are symplectic follows from \eqref{eq:equ2for} in Theorem \ref{thm:morequ}.

{\em Step 3: principality of the actions}. The actions of $\mathcal{G}_{2n}$ and $\mathcal{G}_{2m}$ on $\mathcal{G}_{n+m}$ are obtained by reduction from a Morita equivalence but we still need to verify the surjectivity of the corresponding moment maps $p$ and $q$ as that property is not guaranteed to hold after reduction. In other words, we need to show that we can extend any representation $\rho\in \hom(\Pi_1(\widehat{\Sigma_{0,n}},\widehat{V_{0,n}}),G)$ (with corresponding boundary conditions) to a representation in $\hom(\Pi_1(\widehat{\Sigma_{n+m}},\widehat{V_{n+m}}),G)$ with boundary conditions as in Remark \ref{rem:dec dou sur}. For this purpose, it is enough to consider the following two values of $\rho$ and then show that we can find a representation $\rho'\in\hom(\Pi_1(\widehat{\Sigma_{0,m}},\widehat{V_{0,m}}),G)$ that has those values at a pair of boundary edges and, additionally, that $\rho$ and $\rho'$ together determine a suitably decorated representation in $\hom(\Pi_1(\widehat{\Sigma_{n+m}},\widehat{V_{n+m}}),G)$. The first of these two values is $\rho(e)=g$, where $e$ is the path along which $\widehat{\Sigma_{0,n}}$ and $\widehat{\Sigma_{0,m}}$ are glued to produce $\widehat{\Sigma_{n+m}}$, as indicated in the dashed path in Fig.~\ref{fig:dual}. The second value can be $\rho(e')=b$, where $e'$ is the boundary edge adjacent to $e$, note that $b\in B_+$.
\begin{figure}[H] 
 \begin{tikzpicture}[line cap=round, line join=round, >=stealth, scale=0.5]
  \coordinate (C1) at (-2.5, 0);
  \coordinate (C2) at (0, 0);
  \coordinate (C3) at (2.5, 0);
  \def\r{0.6}
  
  \coordinate (v1) at (0, 2.5);
  \coordinate (v2) at (0, -2.5);
  \coordinate (v3) at (-3.1, 0);
  \coordinate (v4) at (-1.9, 0);
  \coordinate (v5) at (-0.6, 0);
  \coordinate (v6) at (0.6, 0);
  \coordinate (v7) at (1.9, 0);
  \coordinate (v8) at (3.1, 0);

  \draw[thick] (v1) arc (90:270:5 and 2.5) node[midway, left=6pt] {$b$};
  \draw[thick] (v1) arc (90:-90:5 and 2.5) node[midway, right=6pt] {$g$};

  \draw[thick, <-] (v3) arc (180:0:\r) node[midway, above=1pt] {$x$};
  \draw[thick, <-] (v3) arc (180:360:\r) node[midway, below=1pt] {$y$};

  \draw[thick] (v5) arc (180:0:\r) node[midway, above=1pt] {$1$};
  \draw[thick] (v5) arc (180:360:\r) node[midway, below=1pt] {$1$};

  \draw[thick] (v7) arc (180:0:\r) node[midway, above=1pt] {$1$};
  \draw[thick] (v7) arc (180:360:\r) node[midway, below=1pt] {$1$};

  \draw[thick,dashed] (v4) -- (v5) node[midway, below=2pt] {$1$};
  \draw[thick,dashed] (v6) -- (v7) node[midway, below=2pt] {$1$};

  \draw[thick,->,dashed] (v1) -- (v4) node[midway, left=1pt] {$h$};

  \begin{scriptsize}
    \fill (v1) circle (1.5pt) node[above] {};
    \fill (v2) circle (1.5pt) node[below] {};
    \fill (v3) circle (1.5pt) node[left] {};
    \fill (v4) circle (1.5pt) node[above right] {};
    \fill (v5) circle (1.5pt) node[above left] {};
    \fill (v6) circle (1.5pt) node[above right] {};
    \fill (v7) circle (1.5pt) node[above left] {};
    \fill (v8) circle (1.5pt) node[right] {};
  \end{scriptsize}
\end{tikzpicture}
\caption{The values that specify $\rho'\in\hom(\Pi_1(\widehat{\Sigma_{0,m}},\widehat{V_{0,m}}),G)$}
    \label{fig:surgen}
\end{figure}
Consider the generators for $\Pi_1(\widehat{\Sigma_{0,m}},\widehat{V_{0,m}})$ displayed in Fig.~\ref{fig:surgen}. We decorate with $1$ the boundary components supposed to be decorated with $\widetilde{B }_+$, the remaining boundary component imposes the constraint $(x,y)\in G^*$. So the only additional equation that we need to satisfy is $h^{-1}(y^{-1}x)h=bg^{-1}$ for some $h\in G$. But $y^{-1}x$ must lie in the big Gauss cell $N_-TN_+\subset G$ which covers the entirety of $G$ under conjugation. This means that we can find $x,y,h$ that solve the equation $h^{-1}(y^{-1}x)h=bg^{-1}$ as desired. To conclude the argument, note that the action maps may be expressed in terms of the actions of the side groupoids on $\widetilde{F}_{n+m}$ as argued in Proposition \ref{pro:slim}: so the domain of the action map is 
\[ \mathcal{G}_{2m} \times_{\widetilde{F}_{2m}^* } \mathcal{G}_{m+n}\subset(\widetilde{F}_{2m} \times \widetilde{F}_{2m}^* \times \widetilde{F}_{2m}^* \times \widetilde{F}_{2m})\times_{\widetilde{F}_{2m}^*}(\widetilde{F}_{m+n} \times \widetilde{F}_{2n}^* \times \widetilde{F}_{2m}^* \times \widetilde{F}_{m+n})  \]
and each of the factors is defined as a subvariety by equations as in \eqref{eq:eqs}.
 Using the principality of the action $\widetilde{F}_{2m}\times_{F_m} \widetilde{F}_{m+n}\rightarrow \widetilde{F}_{m+n}$ we may then use the isomorphism $\widetilde{F}_{2m}\times_{F_m} \widetilde{F}_{m+n} \rightarrow \widetilde{F}_{m+n}\times_{F_n} \widetilde{F}_{m+n} $ defined by $ (g,z)\mapsto (g\cdot z,z)$ to identify 
\[ \mathcal{G}_{2m} \times_{\widetilde{F}_{2m}^* } \mathcal{G}_{m+n} \cong \mathcal{G}_{m+n} \times_{\widetilde{F}_{2n}^* } \mathcal{G}_{m+n} \subset(\widetilde{F}_{m+n} \times \widetilde{F}_{2n}^* \times \widetilde{F}_{2m}^* \times \widetilde{F}_{m+n})\times_{\widetilde{F}_{2n}^*}(\widetilde{F}_{m+n} \times \widetilde{F}_{2n}^* \times \widetilde{F}_{2m}^* \times \widetilde{F}_{m+n})  \]
as desired, so the result follows. Proceed similarly for the action of $\mathcal{G}_{2n} $.   
\end{proof} 
\subsection{Restriction of the Morita equivalence to symplectic subgroupoids of $\widetilde{F}_{2n}$}\label{subsec:symfol} One of the main results in \cite{conpoigro} is a description of the symplectic foliation of $\widetilde{F}_{2n}$, including the fact that the symplectic leaves passing through the unit submanifold $F_n \subset\widetilde{F}_{2n}$ are symplectic subgroupoids and that in fact all the units of the special configuration Poisson groupoid of flags $\Gamma^{(\mathbf{u},\mathbf{u}^{-1})} \rightrightarrows \mathcal{O}^{\mathbf{u} }  $ lie in a single symplectic leaf. A general fact is that the orbits of a double groupoid that pass through the units of a side groupoid are Lie subgroupoids. To simplify notation, in the course of the following discussion, we omit the unit embeddings in the side groupoids.
\begin{prop}\label{cor:symorb} Take $p\in F_n$, then 
\[ \begin{gathered}  \mathcal{G}_{2n} \cdot^V p =\{x\in \widetilde{F}_{2n}\ |\ \text{there is $g\in \mathcal{G}_{2n}$ such that $\mathtt{t}^V(g)=x$ and $\mathtt{s}^V(g)=\mathtt{u} (p) $}  \}\subset \widetilde{F}_{2n} \\
\mathcal{G}_{2n} \cdot^H p =\{\xi\in \widetilde{F}_{2n}^*\ |\ \text{there is $g\in \mathcal{G}_{2n}$ such that $\mathtt{t}^H(g)=\xi$ and $\mathtt{s}^H(g)=\mathtt{u} (p) $}  \}\subset \widetilde{F}_{2n}^* \end{gathered} \]
are symplectic subgroupoids of the corresponding side groupoids. 
  \end{prop}
\begin{rema} Note that for $p\in F_n$, $\mathcal{G}_{2n} \cdot^V p $ is a groupoid over the orbit $\widetilde{F}_{2n}^*\cdot p$ and $\mathcal{G}_{2n} \cdot^H p$ is a groupoid over $\widetilde{F}_{2n}\cdot p$ but $\widetilde{F}_{2n}$ is a transitive groupoid and so $\widetilde{F}_{2n}\cdot p=F_n$. It follows that $\mathcal{G}_{2n} \cdot^H p \rightrightarrows F_n$ is a symplectic groupoid that integrates the Poisson structure on $F_n$. \end{rema}  
As we shall explain below, the proof of the next proposition specializes to a proof of Proposition \ref{cor:symorb}. 
 
\begin{prop}\label{pro:symorbact} For all $x\in \widetilde{F}_{n+m}$, the symplectic actions of Theorem \ref{thm:symdoumor} restrict to a pair of commuting symplectic actions of $\mathcal{G}_{2n} \cdot^V q_0(x) $ and $\mathcal{G}_{2m} \cdot^V p_0(x)$ on the orbit $\mathcal{G}_{n+m}\cdot x \subset \widetilde{F}_{n+m}$. \end{prop}
\begin{proof} Firstly, the groupoid structures in Proposition \ref{cor:symorb} arise from setting $n=m$ and $x=p\in F_n\subset \F$ in the argument below. Now take $x \in \widetilde{F}_{n+m}$ and $g\in \mathcal{G}_{2m}$ such that $\mathtt{s}^V(g)=p_0(x)$. Suppose there is $g'\in \mathcal{G}_{n+m}$ such that $\mathtt{s}(g')=x$ and $\mathtt{s}(\mathtt{t}^V(g))=p_0(\mathtt{t}(g')) $. We need to show that $\mathtt{t}^V(g)\cdot \mathtt{t}(g')\in \mathcal{G}_{n+m}\cdot x $. Consider $h:=\mathtt{u}^H(\mathtt{s}^H(g)^{-1}p(g'))$, note that $\mathtt{s}^V(h)=\mathtt{t}^V(h)= p_0(x)$. Then $k:=\mathtt{m}^V(g,h)$ is such that $\mathtt{t}^V(k)=\mathtt{t}^V(g) $ and $\mathtt{s}^H (k)=p(g')$. As a consequence, the element $k\cdot g'$ is defined and it satisfies the conditions $\mathtt{t}(k\cdot g')= \mathtt{t}^V(g)\cdot \mathtt{t}(g')$ and $\mathtt{s}(k\cdot g')=x$, thus completing the proof, see Fig.~\ref{fig:resmor}.          
\end{proof}

\begin{wrapfigure}{r}{8cm} 
\begin{tikzpicture}[line cap=round,line join=round,>=triangle 45,x=1cm,y=1cm,scale=0.7]
\clip(-5.388461380851111,-4.303729547683533) rectangle (6.73463735311124,1.9359045012127334);
\fill[line width=1pt,color=black,fill=black,fill opacity=0.10000000149011612] (-3.2310078990054785,1.3811005069157833) -- (-3.24,-0.72) -- (-1.1388994930842171,-0.7289921009945219) -- (-1.129907392089695,1.3721084059212614) -- cycle;
\fill[line width=1pt,color=black,fill=black,fill opacity=0.10000000149011612] (-3.2125614835904406,-1.8857596630873033) -- (-3.212561483590441,-3.8964189433263567) -- (-1.201902203351387,-3.8964189433263567) -- (-1.2019022033513869,-1.8857596630873033) -- cycle;
\fill[line width=1pt,color=black,fill=black,fill opacity=0.10000000149011612] (2.8285123288731711,-0.20329115289316996) -- (2.828512328873171,-2.3955924084275217) -- (5.0208135844075223,-2.3955924084275217) -- (5.020813584407523,-0.20329115289317023) -- cycle;
\draw [line width=1pt,color=black] (-3.2310078990054785,1.3811005069157833)-- (-3.24,-0.72);
\draw [line width=1pt,color=black] (-3.24,-0.72)-- (-1.1388994930842171,-0.7289921009945219);
\draw [line width=1pt,color=black] (-1.1388994930842171,-0.7289921009945219)-- (-1.129907392089695,1.3721084059212614);
\draw [line width=1pt,color=black] (-1.129907392089695,1.3721084059212614)-- (-3.2310078990054785,1.3811005069157833);
\draw [line width=1pt,color=black] (-3.2125614835904406,-1.8857596630873033)-- (-3.212561483590441,-3.8964189433263567);
\draw [line width=1pt,color=black] (-3.212561483590441,-3.8964189433263567)-- (-1.201902203351387,-3.8964189433263567);
\draw [line width=1pt,color=black] (-1.201902203351387,-3.8964189433263567)-- (-1.2019022033513869,-1.8857596630873033);
\draw [line width=1pt,color=black] (-1.2019022033513869,-1.8857596630873033)-- (-3.2125614835904406,-1.8857596630873033);
\draw [line width=1pt,color=black] (2.8285123288731711,-0.20329115289316996)-- (2.828512328873171,-2.3955924084275217);
\draw [line width=1pt,color=black] (2.828512328873171,-2.3955924084275217)-- (5.0208135844075223,-2.3955924084275217);
\draw [line width=1pt,color=black] (5.0208135844075223,-2.3955924084275217)-- (5.020813584407523,-0.20329115289317023);
\draw [line width=1pt,color=black] (5.020813584407523,-0.20329115289317023)-- (2.8285123288731711,-0.20329115289316996);
\begin{scriptsize}
\draw [fill=black] (-3.2310078990054785,1.3811005069157833) circle (1.5pt);
\draw[color=black] (-2.1509281261570234,1.7354230171129545) node {$\mathtt{t}^V(g)  $};
\draw [fill=black] (-3.24,-0.72) circle (1.5pt);
\draw[color=black] (-2.1509281261570234,-0.37507595276713035) node {$p_0(x)$};
\draw [fill=black] (-1.1388994930842171,-0.7289921009945219) circle (1.5pt);
\draw [fill=black] (-1.129907392089695,1.3721084059212614) circle (1.5pt);
\draw [fill=black] (-3.2125614835904406,-1.8857596630873033) circle (1.5pt);
\draw[color=black] (-2.1509281261570234,-1.5366684090577198) node {$p_0(x)$};
\draw [fill=black] (-3.212561483590441,-3.8964189433263567) circle (1.5pt);
\draw[color=black] (-2.1509281261570234,-3.5490046361526844) node {$p_0(x)$};
\draw [fill=black] (-1.201902203351387,-3.8964189433263567) circle (1.5pt);
\draw [fill=black] (-1.2019022033513869,-1.8857596630873033) circle (1.5pt);
\draw [fill=black] (2.8285123288731711,-0.20329115289316996) circle (1.5pt);
\draw[color=black] (3.955546613687174,0.14845867542017752) node {$\mathtt{t}(g') $};
\draw [fill=black] (2.828512328873171,-2.3955924084275217) circle (1.5pt);
\draw[color=black] (3.955546613687174,-2.643842580114174) node {$x$};
\draw [fill=black] (5.0208135844075223,-2.3955924084275217) circle (1.5pt);
\draw [fill=black] (5.020813584407523,-0.20329115289317023) circle (1.5pt);
\draw[color=black] (-2.1672885832878767,0.3883537607383388) node {$g$};
\draw[color=black] (-0.2672885832878767,0.3883537607383388) node {$\mathtt{s}^H(g)  $};
\draw[color=black] (-2.1509281261570234,-2.834656294039775) node {$h$};
\draw[color=black] (0.1509281261570234,-2.834656294039775) node {$\mathtt{s}^H (g)^{-1} p(g')$};
\draw[color=black] (3.955546613687174,-1.2640524094778516) node {$g'$};
\draw[color=black] (2.155546613687174,-1.2640524094778516) node {$p(g')$};
\end{scriptsize}
\end{tikzpicture}
\caption{Restriction of a horizontal Morita equivalence to an action of vertical orbits}\label{fig:resmor}\end{wrapfigure}
\begin{rema} We can see that the above argument extends to any Morita equivalence of double groupoids. However, the actions thus obtained do not have to be principal in general. \end{rema}  
One may deduce from Proposition \ref{cor:symorb} that generalized Schubert cells are contained in symplectic leaves, a fact pointed out in \cite[Thm. 5.1]{conpoigro} (see Corollary \ref{cor:gensch}). However, as it is unclear whether the vertical or horizontal source-fibres of $\mathcal{G}_{2n}$ are connected, it is not immediate that such symplectic leaves are groupoids themselves in our picture.

\section{Symplectic groupoids over configuration Poisson groupoids of flags} In this section, we construct symplectic double groupoids integrating the total configuration Poisson groupoids of flags $\Gamma_{2n} \rightrightarrows F_n$. To start we explain what the general integration strategy is for a decorated moduli space and then we shall show how the strategy specializes to the examples in \S\ref{sec:confla}. This discussion is based on \cite{poigromod,2lagpoi}. 

Consider a decorated moduli space $\mathfrak{M}_G(\Sigma,V)_{H,\mathcal{A}} $ as in \eqref{eq:decmodspa}. Now we apply the observation of \S\ref{subsec:poigromod} to a specific marked surface obtained from $(\Sigma,V)$, which coincides with the {\em topological double of $(\Sigma,V)$} introduced in \cite{symdoumodspa}. Let $(\Sigma_D,V_D)$ be the surface obtained by gluing two copies of $\Sigma$ as follows:
\begin{itemize}
    \item The two copies of $\partial \Sigma$ in the disjoint union $\Sigma \sqcup \Sigma$ are identified via the identity map to produce a closed surface $\Sigma \cup_{\partial \Sigma} \Sigma$. Consequently, any boundary component of $\Sigma$ becomes a closed curve in the interior of the doubled surface.
    \item We perform the real oriented blowup of every point in the image of $V$ within $\Sigma \cup_{\partial \Sigma} \Sigma$. Thus, if the original set $V$ consists of $n$ points on $\partial \Sigma$, the resulting surface $\Sigma_D$ has exactly $n$ boundary components. 
\item Finally, we choose two marked points on each boundary component of $\Sigma_D$, this determines the set of marked points $V_D\subset \partial \Sigma_D$. 
\end{itemize}
The decoration $\mathcal{A} $ of $V$ determines a unique decoration $\mathcal{A}_D$ of the vertices of $(\Sigma_D,V_D)$ which is symmetric with respect to the decomposition of $\Sigma_D$ as the gluing of two copies of $(\Sigma,V)$ as above. Now decorate the boundary of $(\Sigma_D,V_D)$ with the $\mathcal{A}_D $-orbit $\mathcal{L}$ of the unit $1\in G^{E_D}$, where $E_D$ is the set of boundary edges of $(\Sigma_D,V_D)$. 

Then we have a symplectic groupoid structure on 
\begin{align}  \mathfrak{M}_G(\Sigma_D,V_D)_{\mathcal{L},\mathcal{A}_D} \rightrightarrows \mathfrak{M}_G(\Sigma,V)_{G^E,\mathcal{A}}, \label{eq:intdecmodspa} \end{align} 
see \cite[Thm. 3.5, Prop. 3.8]{poigromod} and also \cite{2lagpoi}. Note that, in the decoration of $\mathfrak{M}_G(\Sigma,V)_{G^E,\mathcal{A}}$, the edges are fully decorated with copies of $G$. 
\begin{exa} The surface $(\widehat{\Sigma_k},\widehat{V_k})$ constructed at the beginning of \S\ref{sec:morequ} is the particular case of $(\Sigma_D,V_D)$ that corresponds to $(\Sigma_k,V_k)$ being a disc with $k+1$ marked points. Note that, in general, $\mathcal{L} $ is given by a product of subgroups of $G \times G$ just as in the special case of Remark \ref{rem:dec dou sur}. \end{exa}
\begin{prop} The inclusion $\mathfrak{M}_G(\Sigma,V)_{H,\mathcal{A}}\subset \mathfrak{M}_G(\Sigma,V)_{G^E,\mathcal{A}} $ is a Poisson map and the decorated moduli space $\mathfrak{M}_G(\Sigma,V)_{G^E,\mathcal{A}}$ is in turn integrable by the symplectic groupoid \eqref{eq:intdecmodspa}. \qed \end{prop} 
As a consequence, integrating a specific decorated moduli space $\mathfrak{M}_G(\Sigma,V)_{H,\mathcal{A}} $ amounts to: (I) constructing the larger symplectic groupoid \eqref{eq:intdecmodspa} and then (II) reducing the restriction of \eqref{eq:intdecmodspa} to  $\mathfrak{M}_G(\Sigma,V)_{H,\mathcal{A}} $ (if possible) by its characteristic (or null) foliation as in \cite[\S 9.2,Prop. 8]{craruipoi} to obtain a suitable integration. Such a restricted groupoid is a coisotropic submanifold of $\mathfrak{M}_G(\Sigma_D,V_D)_{\mathcal{L},\mathcal{A}_D}$ and is called an {\em over-symplectic groupoid}, it plays a role in the theory of normal forms around Poisson submanifolds \cite{fermar}. The symplectic groupoid \eqref{eq:intdecmodspa} may be viewed as the fibred product of two 2-shifted Lagrangian morphisms with target $(G \times G)^E$, as explained in \cite{2lagpoi}.

A further simplification of the integration problem for Poisson submanifolds arises when they are Lie-Dirac submanifolds \cite{dirsub, craruipoi}. A submanifold of an integrable Poisson manifold is Lie-Dirac if and only if it is the base manifold of a symplectic subgroupoid of some integration. Consequently, in this setting, rather than performing a reduction of an over-symplectic groupoid, the problem reduces to identifying a subgroupoid inside it that has a trivial intersection with the leaves of its characteristic foliation, thereby becoming a symplectic subgroupoid. This is precisely the situation that is relevant for the present work. 
\subsection{Integration of the total configuration Poisson groupoids of flags} 
Now we shall implement the general strategy for the integration of decorated moduli spaces described above to the $n$-th total configuration Poisson groupoids of flags. In this situation, since we are dealing with decorated moduli spaces which are also Poisson groupoids, there is a further general observation about how to produce a symplectic double groupoid that is an integration of a Poisson subgroupoid of one of its side groupoids. Although such an observation is not fully general, it suffices for our purposes.
\subsubsection{Lie-Dirac subgroupoids of Poisson groupoids and slim symplectic double groupoids} The infinitesimal counterpart of a (smooth or holomorphic) Lie groupoid ${G} \rightrightarrows M$ is its {\em Lie algebroid}, which is the vector bundle $A_G:=\ker T\mathtt{s}|_M $, equipped with the vector bundle map called its {\em anchor} $\mathtt{a}:A_G \rightarrow TM$, which is given by $T\mathtt{t}|_A$ and a Lie bracket on its sheaf of sections, given by extending each local section to a local right-invariant vector field on $G$ and then restricting the Lie bracket of vector fields to such right-invariant extensions \cite{macgen}.

In the case of the side groupoids $\mathcal{A},\mathcal{B} $ of a symplectic double groupoid $\mathcal{D} $, their corresponding Lie algebroids constitute a {\em Manin triple} \cite{liuweixu} or, equivalently, a {\em Lie bialgebroid} \cite{macxu}. Conversely, the main result of \cite{tracou} is that, if the Lie algebroids $A,B$ in a Manin triple over $M$ have transverse anchor maps to $TM$, we can construct $\mathcal{D} $ under mild assumptions. These assumptions include the existence of Lie groupoids $\mathcal{A} $ and $\mathcal{B} $ that are respective integrations of $A$ and $B$. In such a situation, we can produce a double Lie groupoid as a submanifold
\begin{align}  \mathcal{D} \subset \{(x,v,u,y) \in \mathcal{A} \times \mathcal{B}\times \mathcal{B} \times \mathcal{A}| \, \mathtt{t}(x)=\mathtt{t}  (u),\, \mathtt{s}(x)=\mathtt{t}  (v),\, \mathtt{s}(u)=\mathtt{t}  (y),\, \mathtt{s}(v)=\mathtt{s}  (y) \}; \label{eq:luwei1}\end{align} 
with the two multiplications defined by 
\begin{equation} \begin{gathered} \mathtt{m}^H ((x,v,u,y) ,(x',w,v,y') )=( xx',w,u, yy'),\\
\mathtt{m}^V ((x,v,u,y) ,(y,v',u',z) )=(x,vv',uu',z);\end{gathered}\label{eq:luweimul}
\end{equation}  
see \cite[Thm. 4.10]{tracou}. In what follows, we shall say that a double groupoid $\mathcal{D}$ is {\em slim} if it is embedded in the product of its side groupoids via the map
\[ d\mapsto (\mathtt{t}^V(d),\mathtt{s}^H(d),\mathtt{t}^H(d),\mathtt{s}^V(d)   ). \]
 
\begin{rema} Providing that $\mathcal{A},\mathcal{B}  $ are equipped with suitable 2-forms (which can be produced using van Est integration \cite{2lagpoi}), we can define a symplectic form on $\mathcal{D} $ which makes it into a symplectic double groupoid \cite[Thm. 4.10]{tracou}. Such a symplectic form also depends on the choice of a splitting of an exact Courant algebroid \cite[Thm. 3.5, Thm. 4.10]{tracou}. Since these methods are transcendental, it is not obvious that the resulting symplectic form should be algebraic, even if $\mathcal{A} $ and $\mathcal{B}$ are. This contrasts with the symplectic forms derived from moduli-theoretic considerations, which are manifestly algebraic.\end{rema}

Suppose that $\mathcal{D} $ as in \eqref{eq:luwei1} is a slim symplectic double groupoid with symplectic form $\omega $. Recall that $B$ can be identified with $A^*$ via the pairing given by 
\begin{equation}  (u,\xi)\mapsto \omega(T\mathtt{u}^V(u),T\mathtt{u}^H(\xi)) \label{eq:pairing} \end{equation} 
for all $(u,\xi)\in A\oplus B$ \cite[Prop. 2.6]{intgk}. Suppose that $\widetilde{\mathcal{A}}\subset \mathcal{A}$ is a saturated Poisson subgroupoid, meaning that it is a union of $\mathcal{D}$-orbits, then $\mathcal{D}|_{\widetilde{\mathcal{A} } } $ is a Lie subgroupoid. The corresponding Lie algebroid $\widetilde{A}\subset A$ is a Lie subalgebroid over a submanifold $N \subset M$ such that the annihilator $ \text{Ann}(\widetilde{A})\subset B|_N$ is a bundle of ideals lying in $\ker (\mathtt{a}_B )$, where $\mathtt{a}_B:B \rightarrow TM$ is the anchor. Suppose that $ \text{Ann}(\widetilde{A})$ is integrable by a normal source-connected Lie subgroupoid $\mathcal{N} \subset \mathcal{B}|_N $ that we shall call an {\em annihilator subgroupoid of $\widetilde{\mathcal{A} } $}. Then we can define the following double Lie subgroupoid of \eqref{eq:luwei1}:
\begin{align}  \mathcal{K} =\{(x,u,v,x) \in \widetilde{\mathcal{A}} \times_N \mathcal{N}\times \mathcal{N} \times_N \widetilde{\mathcal{A}}| \, (x,u,v,x)\in \mathcal{D}  \}. \label{eq:luwei3}\end{align}
\begin{prop}\label{pro:ker} The Lie algebroid  $\mathfrak{K}   $ of $\mathcal{K}$ over $\widetilde{A}$ coincides with $\ker \iota^*\omega|_{\widetilde{\mathcal{A} }} $, where $\iota:\mathcal{D}|_{\widetilde{\mathcal{A}  } } \hookrightarrow \mathcal{D} $ is the inclusion.
\end{prop}
\begin{proof} Using the fact that $\mathcal{D}|_{\widetilde{\mathcal{A}  } }$ is still a double groupoid, the pullback of the symplectic form to it is still doubly multiplicative and so it vanishes on $\mathcal{K} $, because $\mathcal{K} $ consists of vertical isotropy groups. To verify that $\ker \iota^*\omega $ is given by $\mathfrak{K} $, we can just check the claim at $N \subset \mathcal{D}|_{\widetilde{\mathcal{A}  } }$ and that follows from the fact that $\omega|_M $ determines the duality pairing between $A$ and $B$. \end{proof}
\begin{prop}\label{pro:quopoisub} Suppose that the right translation action of $\mathcal{N} $ on $\mathcal{B}|_N$ is principal with quotient $\widetilde{\mathcal{B} } \rightrightarrows N$ and there is a multiplicative section $\sigma: \widetilde{\mathcal{B} } \rightarrow \mathcal{B}  $ of the quotient map. If $(\mathtt{t}^H, \mathtt{s}^H): \mathcal{D}|_{\widetilde{\mathcal{A} } } \rightarrow \mathcal{B}|_N^2 $ has clean intersection with the diagonal embedding of $\sigma(\widetilde{\mathcal{B}})$, then the double Lie groupoid 
\[ \widetilde{\mathcal{D} }:= \{(x,v,u,y) \in \mathcal{D}|_{\widetilde{\mathcal{A} } }| \, u,v\in \sigma(\widetilde{\mathcal{B} } ) \} \] 
is a symplectic double subgroupoid of $\mathcal{D} $ that integrates the Poisson groupoid $\widetilde{\mathcal{A} }$. \end{prop}
\begin{rema} In the terminology of \cite{craruipoi}, $\widetilde{\mathcal{A} }$ is then a {\em Lie-Dirac submanifold}, being the infinitesimal counterpart of a symplectic subgroupoid, see \cite[Thm. 9]{craruipoi}. Lie-Dirac submanifolds were introduced in \cite{dirsub}, where they were called Dirac submanifolds. Interestingly, not every Poisson submanifold is a Lie-Dirac submanifold \cite{craruipoi,dirsub}. It is easy to find situations where the anchors of $\widetilde{A}$ and $\widetilde{B} $ are not transverse, hence Proposition \ref{pro:quopoisub} could lead to integrating Manin triples for which \cite[Thm. 4.10]{tracou} does not apply. \end{rema} 

\begin{proof}[Proof of Proposition \ref{pro:quopoisub}] By construction, the Lie algebroid of $\widetilde{\mathcal{D} }$ over $\widetilde{\mathcal{A} }$ has trivial intersection with $\mathfrak{K} $, we only need to show that $\omega $ restricts to a nondegenerate 2-form along $\widetilde{\mathcal{A} }\subset \widetilde{\mathcal{D} }$; in fact, by multiplicativity it has to be nondegenerate globally. Since the duality pairing between the Lie algebroids $\widetilde{A} $ and $B|_N/ \text{Ann}(\widetilde{A} )$ is the one induced by $\omega$, $\widetilde{\mathcal{D} }$ integrates $\widetilde{\mathcal{A} }$ as desired. Note that $T \widetilde{\mathcal{D}}+\mathfrak{k}  =T (\mathcal{D}|_{\widetilde{A} }) $ along $\widetilde{A}$, since $\sigma(\widetilde{\mathcal{B} } )$ is a slice for the quotient map $\mathcal{B}|_N \rightarrow  \widetilde{\mathcal{B} } $. Since Proposition \ref{pro:ker} tells us that $\mathfrak{k}=\ker \iota^*\omega|_{\widetilde{\mathcal{A} } } $, $\omega $ is nondegenerate when restricted to $T \widetilde{\mathcal{D} }|_{\widetilde{\mathcal{A} } } $, and then its multiplicativity implies that it is nondegenerate globally. \end{proof}
\subsubsection{The $n$th total configuration Poisson groupoid of flags $\Gamma_{2n} \rightrightarrows F_n$ as a Lie-Dirac submanifold of $\widetilde{F}_{2n} \rightrightarrows F_n$} Now we are going to use our integration of $\widetilde{F}_{2n} \rightrightarrows F_n$ to obtain an integration of $\Gamma_{2n} \rightrightarrows F_n$. 
\begin{thm}\label{thm:inttotcon} We have that $\Gamma_{2n} \rightrightarrows F_n$ is a Lie-Dirac submanifold of $\widetilde{F}_{2n} \rightrightarrows F_n$ which is integrable by a complex algebraic symplectic double subgroupoid of $\mathcal{G}_{2n}$. \end{thm}

To prove this result we combine the description of the groupoid structure on $\widetilde{F}_{2n}^* \rightrightarrows F_n$ in terms of explicit formulas given in Propositions \ref{pro:expduapoi} and \ref{pro:slim} with Proposition \ref{pro:quopoisub}.
\begin{lem}\label{lem:ann} The annihilator subgroupoid $\mathcal{N}_{2n} \subset \widetilde{F}_{2n}^*  $ of the total configuration groupoid of flags $\Gamma_{2n} \rightrightarrows F_n$ is isomorphic to the groupoid $N_- \times F_n \rightrightarrows F_n$, viewed as a bundle of Lie groups, and the quotient map $\widetilde{F}_{2n}^* \rightarrow \widetilde{F}_{2n}^*/\mathcal{N}_{2n}$ has a section given by the canonical inclusion of the wide subgroupoid of $\widetilde{F}_{2n}^*$:
\[ \Gamma_{2n}^*:=[(B_+ \times N_+^{n-1} )\times_{B_+^n} G^n] \subset  \widetilde{F}_{2n}^*,\] 
where we view $\widetilde{F}_{2n}^*$ as in Proposition \ref{pro:expduapoi}. \end{lem}
\begin{proof} 
For each of the vertices $v_i$ in $V_{0,n}$, denote by $\widehat{v}_i$ and $\widehat{v}_i'$ its two images in $\widehat{\Sigma}_{0,n}$ and let $\iota_i:\widehat{\Sigma}_{0,n} \hookrightarrow \widehat{\Sigma}_{2n}$ be the two inclusions. Suppose that $v_1$ is decorated with the Lie algebra $\mathfrak{g}^* $ as in \eqref{eq:dec halfsur}. Take a triangulation $\mathcal{T} $ of $(\widehat{\Sigma}_{2n},\widehat{V}_{2n} )$ such that $\{\iota_1(\widehat{v}_1),\iota_1(\widehat{v}_1'),\iota_2(\widehat{v}_1')\}$ and $\{\iota_2(\widehat{v}_1),\iota_2(\widehat{v}_1'),\iota_1(\widehat{v}_1)\}$ constitute the vertices of two adjacent triangles, respectively denoted by $T_1$ and $T_2$ in what follows, see Fig.~\ref{fig:triang}. The reflection with respect to the dotted vertical axis of symmetry in Fig.~\ref{fig:triang} exchanges the images of $\iota_1$ and $\iota_2$. 
\begin{figure}[H] 
\begin{tikzpicture}[semithick, line cap=round, line join=round, scale=0.8]
  \def\r{0.2} 

  \coordinate (ET) at (0, 3);   
  \coordinate (EB) at (0, -3);  

  \coordinate (C1T) at (-3, \r);  
  \coordinate (C1B) at (-3, -\r); 
  \coordinate (C2T) at (-1, \r);  
  \coordinate (C2B) at (-1, -\r); 
  \coordinate (C3T) at (1, \r);   
  \coordinate (C3B) at (1, -\r);  
  \coordinate (C4T) at (3, \r);   
  \coordinate (C4B) at (3, -\r);  

  \draw[thick] (0,0) ellipse (5 and 3);
  \foreach \x in {-3,-1,1,3} {
    \draw[thick] (\x,0) circle (\r);
  }

  \draw[dotted, thin] (ET) -- (EB);

  \draw[dashed] (ET) to[bend right=30] (C1T); 
  \draw[dashed] (ET) to[bend right=10] (C2T); 
  \draw[dashed] (ET) to[bend left=10] (C3T);  
  \draw[dashed] (ET) to[bend left=30] (C4T);  

  \draw[dashed] (EB) to[bend left=30] (C1B);  
  \draw[dashed] (EB) to[bend left=10] (C2B);  
  \draw[dashed] (EB) to[bend right=10] (C3B); 
  \draw[dashed] (EB) to[bend right=30] (C4B); 

  \draw[dashed] (EB) to[out=175,in=190, looseness=1.5] (C1T);  
  \draw[dashed] (EB) to[out=5, in=350, looseness=1.5] (C4T); 

  \draw[dashed] (C1T) to[bend left=20] (C2T);
  \draw[dashed] (C1B) to[bend right=20] (C2B);
  
  \draw[dashed] (C2T) to[bend left=20] node[midway, above] {$b_1$} (C3T);
  \draw[dashed] (C2B) to[bend right=20] node[midway, below] {$b_2$} (C3B);
  
  \draw[dashed] (C3T) to[bend left=20] (C4T);
  \draw[dashed] (C3B) to[bend right=20] (C4B);

  \draw[dashed] (C1T) to[bend left=10] (C2B);
  \draw[dashed] (C3T) to[bend left=10] (C4B);
  \draw[dashed] (C2B) to[bend right=10] (C3T);

  \begin{scriptsize}
    \fill (ET) circle (1.5pt) node[above=4pt] {$\iota_1(\widehat{v}_3')=\iota_2(\widehat{v}_3')$};
    \fill (EB) circle (1.5pt) node[below=4pt] {$\iota_1(\widehat{v}_3)=\iota_2(\widehat{v}_3)$};
    
    \fill (C1T) circle (1pt) node[above left=4pt, anchor=east] {$\iota_1(\widehat{v}_2')$};
    \fill (C1B) circle (1pt) node[below left=4pt, anchor=east] {$\iota_1(\widehat{v}_2)$};
    
    \fill (C2T) circle (1pt) node[above left=2pt] {$\iota_1(\widehat{v}_1')$};
    \fill (C2B) circle (1pt) node[below left=2pt] {$\iota_1(\widehat{v}_1)$};
    
    \fill (C3T) circle (1pt) node[above right=2pt] {$\iota_2(\widehat{v}_1')$};
    \fill (C3B) circle (1pt) node[below right=2pt] {$\iota_2(\widehat{v}_1)$};
    
    \fill (C4T) circle (1pt) node[above right=4pt, anchor=west] {$\iota_2(\widehat{v}_2')$};
    \fill (C4B) circle (1pt) node[below right=4pt, anchor=west] {$\iota_2(\widehat{v}_2)$};
  \end{scriptsize}
\end{tikzpicture}
\caption{The dashed edges together with the boundary edges define a triangulation of $\widehat{\Sigma}_{4} $}
    \label{fig:triang}
\end{figure}
The Lie algebroid $\nu_{2n}:=\mathfrak{n}_- \times F_n$ of $\mathcal{N}_{2n}$ is embedded in the action Lie algebroid $(\mathfrak{g}^* \times \mathfrak{n}_+^{n-1}) \ltimes F_n $ of $\widetilde{F}_{2n}^*$ via the inclusion $\mathfrak{n}_-\subset \mathfrak{g}^*$. In order to compute the pairing between $\nu_{2n}$ and the Lie algebroid $\gamma_{2n}$ of $\Gamma_{2n} \rightrightarrows F_n$, we take the symplectic form $\Omega_{2n}$ on $\mathcal{G}_{2n}$ and apply it to $ (u,\xi)\in \gamma_{2n}\oplus \nu_{2n}  $ embedded in $T \mathcal{G}_{2n}$ as in \eqref{eq:pairing}. According to \S\ref{subsec:mom map surf}, such a symplectic form is determined by embedding $\hom(\pi_1(\widehat{\Sigma}_{2n},\widehat{V}_{2n} ),G )\subset (G^2)^{k}$ using the orientation of the edges of $\mathcal{T} $ inherited from $\widehat{\Sigma}_{2n} $ and then pulling back $\omega $ as in \eqref{eq:polwie} to each $G^2$ factor in $G^{2k}$, where $k$ is the number of triangles in $\mathcal{T}$. It follows then that the only triangles contributing a potentially non-zero term in $\Omega_{2n}(T\mathtt{u}^V(u),T\mathtt{u}^H(\xi)  )$ are $T_1 $ and $T_2$. But by construction, the vertical boundary edges joining $\iota_i(\widehat{v}_1)$ and $\iota_i(\widehat{v}_1')$ in Fig.~\ref{fig:triang} are decorated with $N_-$. On the other hand, the edges joining $\iota_1(\widehat{v}_1)$ with $\iota_2(\widehat{v}_1)$ (and $\iota_1(\widehat{v}_1')$ with $\iota_2(\widehat{v}_1')$) are decorated with $B_-$, these edges are labelled with $b_1,b_2$ in Fig.~\ref{fig:triang}. Since $\omega $ vanishes on $B_- \times N_-\subset G^2$, we have then that $\Omega_{2n}(T\mathtt{u}^V(u),T\mathtt{u}^H(\xi)  )=0$. So $\nu_{2n}$ lies in the annihilator of $\gamma_{2n}$ but since it has the same rank it must coincide with it. The last claim follows directly from Proposition \ref{pro:expduapoi}. 
 \end{proof}  
\begin{proof}[Proof of Theorem \ref{thm:inttotcon}] According to Prop.~\ref{pro:slim}, $\mathcal{G}_{2n}$ is a slim double groupoid with algebraic structure maps, so the result follows from applying  Prop.~\ref{pro:quopoisub} to $\Gamma_{2n}^*$. The smoothness of the restricted double groupoid follows directly from considering generators for the fundamental groupoids (or from observing that the foliations of $\Gamma_{2n}$ and $\Gamma_{2n}^* $ are still transverse). \end{proof}
We shall denote by $\mathcal{H}_{2n}$ (with sides $\Gamma_{2n}$ and $\Gamma_{2n}^*$) the resulting symplectic double groupoid from Theorem \ref{thm:inttotcon}. It follows immediately from Theorem \ref{thm:inttotcon} that the following result holds.
\begin{coro}\label{thm:insspe} If $\Gamma^{(\mathbf{u},\mathbf{u}^{-1})  } \rightrightarrows \mathcal{O}^{\mathbf{u} }$ is open in $\Gamma_{2n}$, then it is integrable by a complex algebraic double subgroupoid of $\mathcal{H}_{2n}$. \qed\end{coro}
\subsection{Restriction of the Morita equivalence to $\mathcal{H}_{2n}$} We have the following restricted version of Theorem \ref{thm:symdoumor}
. As demonstrated in the proof of Proposition \ref{pro:slim}, there is a natural inclusion $\mathcal{G}_{m+n}\subset \widetilde{F}_{m+n} \times \widetilde{F}_{2n}^* \times \widetilde{F}_{2m}^* \times \widetilde{F}_{m+n}$, so we may define $\mathcal{H}_{m+n} $ as the intersection 
\[ \mathcal{H}_{n+m}= \mathcal{G}_{m+n}\cap\left( \Gamma_{m+n} \times \Gamma_{2n}^* \times \Gamma_{2m}^* \times \Gamma_{m+n}\right),\]
and we have that it is a symplectic subgroupoid of $\mathcal{G}_{n+m} $ over $\Gamma_{n+m}$ that serves as a symplectic bimodule between $\mathcal{H}_{2m} $ and $\mathcal{H}_{2n} $.
\begin{thm}\label{thm:resmorequ} The Morita equivalence of Theorem \ref{thm:symdoumor} restricts to a symplectic Morita equivalence between $\mathcal{H}_{2m}$  and $\mathcal{H}_{2n}$. \end{thm}      
\begin{proof} By the description of the groupoid structure on $\mathcal{G}_{n+m} $ given in the proof of Proposition \ref{pro:slim}, it follows that $\mathcal{H}_{n+m}$ is preserved by the actions of $\mathcal{H}_{2n},\mathcal{H}_{2m} $. To show that the symplectic form on $\mathcal{G}_{n+m} $ restricts to a symplectic form on $\mathcal{H}_{n+m} $ we proceed as follows. We have that the Lie groupoid structures together with the actions determine a groupoid structure
\[ \mathcal{H}_{2n}\sqcup \mathcal{H}_{n+m} \sqcup \mathcal{H}_{2m} \rightrightarrows \widetilde{F}^*_{2n} \sqcup \widetilde{F}^*_{2m},  \]
as in \cite[Prop. 6.4]{baigua}. Since the 2-form inherited by such a groupoid is multiplicative and nondegenerate along the unit inclusion according to Lemma \ref{lem:ann}, it is nondegenerate globally, in particular, nondegenerate on $\mathcal{H}_{n+m}$.
 
To show that the restricted moment maps $\mathcal{H}_{n+m} \rightarrow \Gamma^*_{2n},\Gamma^*_{2m}$ are still surjective, we proceed as in Step 3 of the proof of Theorem \ref{thm:symdoumor} and we use the fact that every element in $G$ lies in a Borel subgroup and all Borel subgroups are conjugate. Principality follows from using the expression of the double groupoid actions in terms the actions of $\Gamma_{2m}$ and $\Gamma_{2n} $ on $\Gamma_{m+n}$, which are principal. Similarly to the proof of Theorem \ref{thm:symdoumor}, the identification
\[\mathcal{H}_{2m} \times_{\Gamma_{2m}^*} \mathcal{H}_{m+n} \rightarrow \mathcal{H}_{m+n} \times_{\Gamma_{2n}^*} \mathcal{H}_{m+n}, \qquad (h,z)\mapsto (h\cdot z,z);  \]
may be expressed in terms of the corresponding identification for the $\Gamma_{2m}$-action on $\Gamma_{m+n}$. The case of $\mathcal{H}_{2n}$ is treated analogously.  
 \end{proof} 
\subsection{Toric actions on $\mathcal{G}_{2n}$ from quasi-hamiltonian actions on fission spaces} There are canonical $T=B_+\cap B_-$-actions on $\widetilde{F}_m$ and $F_m$ given by
\begin{align} t\cdot [g_1,\dots,g_m]=[tg_1,g_2,\dots,g_m] ; \label{eq:tact} \end{align} 
for all $t\in T$ and all $[g_i]\in \widetilde{F}_{m}$ or $[g_i]\in F_m$. The $T$-orbit of a symplectic leaf inside each of these spaces is called a {\em $T$-leaf} in \cite{conpoigro} and we have the decomposition of $\Gamma_{2n}$ into a finite number of $T$-leaves. In particular, it is proven in \cite{conpoigro} that each special configuration Poisson groupoid of flags $\Gamma^{(\mathbf{u},\mathbf{u}^{-1})  } \rightrightarrows \mathcal{O}^{\mathbf{u} }$ is a $T$-leaf in $\Gamma_{2n}$. The $T$-leaves in $\Gamma_{2n}$ are isomorphic to generalized double Bruhat cells and hence, if $G$ is simply-connected or of adjoint type, they are isomorphic to decorated double Bott-Samelson cells, which are cluster varieties \cite{cludoubot}. Due to these remarkable facts, it is desirable to lift the $T$-actions of \eqref{eq:tact} to the symplectic double groupoids $\mathcal{G}_{2n}$ while providing a purely moduli-theoretic interpretation of them. In fact, the $T$-action lifted to $\mathcal{G}_{2n}$ may be viewed as coming from the theory of wild character varieties \cite{geobrasto,quahammer,fisspa}.

This interpretation follows in a rather straightforward fashion from results in the literature, so we shall be brief. The {\em fission space} ${}_G \mathcal{A}_T=G \times G^*\subset G \times(B_+ \times B_-)$ is a quasi-hamiltonian $G \times T$-space with moment map and action as follows:
\[ \mu_{{}_G \mathcal{A}_T}(c,x,y)=(c^{-1}y^{-1}xc,t^{-1}),\qquad (g,s)\cdot (c,x,y)=(scg^{-1},sxs^{-1},sys^{-1}), \qquad t=\text{pr}_T(x);  \]
for all $(c,x,y)\in {}_G \mathcal{A}_T$ and $(g,s)\in G \times T$, the quasi-symplectic 2-form is as in \cite[Eq. (8)]{quahammer} (the corresponding space $\widetilde{\mathcal{C} }$ therein differs by a covering map from this one):
\[ \omega_{{}_G \mathcal{A}_T}|_{(c,x,y)}=\frac{1}{2}\left[ \langle (yc)^*\theta^r,(xc)^*\theta^r \rangle+\langle (yc)^*\theta^l,c^*\theta^l \rangle- \langle (xc)^*\theta^l,c^*\theta^l \rangle \right]. \]
The quasi-hamiltonian space ${}_G \mathcal{A}_T$ models the generalized monodromy data of meromorphic connections on a disc with a single pole of order two at the origin (with a regular semisimple irregular type). We may interpret ${{}_G \mathcal{A}_T}$ as the decorated moduli space of the cylinder with two marked points on  one component and one on the other, with the component labelled with two points being decorated with $G^*$ and the other one with $G$, see Fig.~\ref{fig:fusdou}. 

Recall that, given two quasi-hamiltonian $G$-spaces $M_1$, $M_2$, we can take their fusion product $M_1\ast M_2$ and then reduce at level $1\in G$:
\[ \mu_i:M_i \rightarrow G, \qquad \mu_1\mu_2: M_1\ast M_2=M_1 \times M_2 \rightarrow G, \qquad [M_1\ast M_2]\sslash G:=(\mu_1\mu_2)^{-1}(1)/G. \]
If $M_1$ and $M_2$ are representation varieties, this last operation corresponds to taking the representation variety of the surface obtained by gluing the original surfaces along a boundary component. This is one formulation of the principle according to which {\em gluing equals reduction} \cite{alemeimal}.
 
Now consider the marked surface $(\widehat{\Sigma_{2n}},\widehat{V_{2n}}')$, where $\widehat{\Sigma}_{2n}$ is as in \S\ref{sec:morequ} but $\widehat{V_{2n}}'$ is given by removing two marked points from $\widehat{V_{2n}}$. Then by decorating all the boundary components of $(\widehat{\Sigma_{2n}},\widehat{V_{2n}}')$ that have two marked points on them with the group $\widetilde{B}_+$ just as in \S\ref{sec:morequ}, we obtain a submanifold
\[ \mathcal{M}_{2n}'\hookrightarrow \hom(\Pi_1(\widehat{\Sigma_{2n}},\widehat{V_{2n}}'),G) \] 
that determines a quasi-hamiltonian $G^2$-space $\mathcal{M}_{2n}=\mathcal{M}'_{2n}/B_+^{2n-1}$ with the restricted action and 2-form inherited from $\hom(\Pi_1(\widehat{\Sigma_{2n}},\widehat{V_{2n}}'),G)$. Here the $B_+^{2n-1}$-action is the one induced by the decoration. In other words, we are performing partial reduction with respect to the decoration of all but two of the boundary components of $(\widehat{\Sigma_{2n}},\widehat{V_{2n}}')$. 
\begin{lem}\label{lem:grofus} We have that $\mathcal{G}_{2n}$ is symplectomorphic to the reduced fusion product $[({}_G \mathcal{A}_T\times  \overline{{}_G \mathcal{A}_T} )\ast \mathcal{M}_{2n} ]\sslash G^2 $, where the fusion is with respect to the $G^2$-actions and $\overline{{}_G \mathcal{A}_T} $ denotes the opposite quasi-hamiltonian $G \times T$-space.   \end{lem} 
\begin{proof} This follows immediately from the fact that one can sew and then decorate different boundary arcs in any order, in the language of \cite{quisur2}. As we discussed, ${{}_G \mathcal{A}_T}$ may be viewed as the moduli space of a decorated cylinder, hence ${}_G \mathcal{A}_T\times  \overline{{}_G \mathcal{A}_T} $ corresponds to a disjoint union of two such cylinders and $({}_G \mathcal{A}_T\times  \overline{{}_G \mathcal{A}_T} )\ast \mathcal{M}_{2n}  $ is the decorated moduli space of a sphere with $2n+3$ holes, this identification is compatible with the 2-forms, just as in \cite[Prop. 7]{quahammer}. Reduction by $G^2$ removes the two additional unmarked boundary components, see Fig.~\ref{fig:fusdou}. \end{proof} 
\begin{figure}[H] 
\begin{tikzpicture}[line width=0.6pt]

\begin{scope}[shift={(0,0)}]
    \node at (0.3, 1.6) {${{}_G \mathcal{A}_T}$}; 
    \draw (0,0.8) ellipse (0.15 and 0.3);
    \draw (0.6,0.8) ellipse (0.15 and 0.3);
    \draw (0,1.1) -- (0.6,1.1) (0,0.5) -- (0.6,0.5);
    
    \draw (0,-0.8) ellipse (0.15 and 0.3);
    \draw (0.6,-0.8) ellipse (0.15 and 0.3);
    \draw (0,-0.5) -- (0.6,-0.5) (0,-1.1) -- (0.6,-1.1);
    
    \node at (0.3, -1.6) {$\overline{{}_G \mathcal{A}_T}$};
\end{scope}

\begin{scope}[shift={(2.5,0)}]
    \node at (1.5, 1.8) {$\mathcal{M}_{4}$}; 

    \draw (0,0.8) ellipse (0.15 and 0.3);
    \draw (0,-0.8) ellipse (0.15 and 0.3);
    
    \draw[dashed, gray!60] (1.5,0) ellipse (0.15 and 0.3);

    \draw (0,1.1) to[out=0,in=120] (1.5,0.3);
    \draw (0,-1.1) to[out=0,in=-120] (1.5,-0.3);
    \draw (0,0.5) to[out=0,in=0,looseness=1.5] (0,-0.5);

    \begin{scope}[shift={(1.5,0)}]
        \draw (1.5,0) ellipse (0.15 and 0.3) node[right=6pt] {$\widetilde{B}_+ $};
        \draw (0.75,0.75) ellipse (0.3 and 0.15) node[above=4pt] {$\widetilde{B}_+ $};
        \draw (0.75,-0.75) ellipse (0.3 and 0.15) node[below=6pt] {$\widetilde{B}_+ $};
        
        \draw (0,0.3) to[out=45,in=180] (0.45,0.75);
        \draw (1.05,0.75) to[out=0,in=135] (1.5,0.3);
        \draw (1.5,-0.3) to[out=-135,in=0] (1.05,-0.75);
        \draw (0.45,-0.75) to[out=180,in=-45] (0,-0.3);
    \end{scope}
\end{scope}

\begin{scope}[shift={(8.0,0)}]
    \node at (1.15, 1.8) {$\mathcal{G}_{4} $}; 
    \draw (0,0.8) ellipse (0.15 and 0.3) node[left=6pt] {$G^*$};
    \draw (0,-0.8) ellipse (0.15 and 0.3) node[left=6pt] {$G^*$};
    \draw (1.2,1.2) ellipse (0.15 and 0.3) node[below=6pt] {$\widetilde{B}_+ $};
    \draw (1.2,-1.2) ellipse (0.15 and 0.3) node[above=6pt] {$\widetilde{B}_+ $};
    \draw (2.3,0) ellipse (0.15 and 0.3) node[right=6pt] {$\widetilde{B}_+ $};
    
    \draw (0,1.1) to[out=10,in=190] (1.2,1.5);
    \draw (1.2,1.5) to[out=0,in=120] (2.3,0.3);
    \draw (2.3,-0.3) to[out=-120,in=0] (1.2,-1.5);
    \draw (1.2,-1.5) to[out=170,in=-10] (0,-1.1);
    
    \draw (0,0.5) to[out=0,in=0,looseness=1.8] (0,-0.5); 
\end{scope}

\end{tikzpicture}
\caption{Gluing both cylinders with the unmarked boundary circles gives rise to the surface on the right which underlies $\mathcal{G}_{4}\cong [({}_G \mathcal{A}_T\times  \overline{{}_G \mathcal{A}_T} )\ast \mathcal{M}_{4} ]\sslash G^2 $}
    \label{fig:fusdou}
\end{figure}  
\begin{thm}\label{thm:lifact} The $T^2$-action on $\mathcal{G}_{2n}$ given by the identification of Lemma \ref{lem:grofus} is an action by automorphisms with respect to the vertical groupoid structure $\mathcal{G}_{2n} \rightrightarrows \widetilde{F}_{2n}$ and it is a multiplicative action with respect to the horizontal groupoid structure $\mathcal{G}_{2n} \rightrightarrows \widetilde{F}_{2n}^*$. Moreover, such an action lifts the $T^2$-action on $\widetilde{F}_{2n}$ given by  
\begin{align}  (t_1,t_2)\cdot [g_1,\dots,g_{2n}]=[t_1g_1,g_2,\dots,g_{2n}t_2^{-1}] \label{eq:t2act} \end{align}  
and it has a $T^2$-valued moment map that makes $\mathcal{G}_{2n}$ into a quasi-hamiltonian $T^2$-space.
\end{thm} 
\begin{proof} To clarify our terminology, the $T^2$-action is multiplicative in the sense that: (1) it is a groupoid morphism from the direct product groupoid
\[ (T^2 \rightrightarrows T) \times (\mathcal{G}_{2n} \rightrightarrows \widetilde{F}^*_{2n}  ) \rightarrow (\mathcal{G}_{2n} \rightrightarrows \widetilde{F}^*_{2n}  ); \]
where $T^2 \rightrightarrows T$ is regarded as a pair groupoid; (2) moreover, such an action is an action by groupoid automorphisms with respect to the vertical groupoid structure $\mathcal{G}_{2n} \rightrightarrows \widetilde{F}_{2n}$.

The existence and quasi-hamiltonian nature of the moment map $\mathcal{G}_{2n} \rightarrow T^2$ follows from the identification $\mathcal{G}_{2n}\cong [({}_G \mathcal{A}_T\times  \overline{{}_G \mathcal{A}_T} )\ast \mathcal{M}_{2n} ]\sslash G^2 $ of Lemma \ref{lem:grofus}. In fact, such a reduced fusion product inherits a residual $T^2$-valued moment map from ${{}_G \mathcal{A}_T}\times  \overline{{}_G \mathcal{A}_T}$, by the standard properties of the fusion product. To show that the corresponding  $T^2$-action lifts \eqref{eq:t2act} and is multiplicative, we proceed as follows. We identify $\widetilde{F}^*_{2n} \rightrightarrows F_n$ with the groupoid $(G^* \times N_+^{n-1}) \times_{B_+^{n}} G^n     \rightrightarrows F_n$, see Proposition \ref{pro:expduapoi}. Then there is a $T$-action on $(G^* \times N_+^{n-1}) \times_{B_+^{n}} G^n     \rightrightarrows F_n$ by automorphisms given by
\[ t\cdot [(x,y),(n_1,\dots, n_{n-1}),(g_1,g_2,\dots,g_n)]=[(txt^{-1},tyt^{-1}),(n_1,\dots,n_{n-1}),(tg_1,g_2,\dots,g_n)]; \] 
for all $(t, [(x,y),(n_1,\dots, n_{n-1}),(g_1,g_2,\dots,g_n)])\in T \times \widetilde{F}_{2n} $. Then the desired $T^2$-action on $\mathcal{G}_{2n}$ is obtained by viewing it as a slim double groupoid $\mathcal{G}_{2n}\subset \widetilde{F}_{2n} \times \widetilde{F}_{2n}^* \times \widetilde{F}_{2n}^* \times \widetilde{F}_{2n} $, see Proposition \ref{pro:slim}. In fact, we may define a $T$-action as above on each of the factors $\widetilde{F}^*_{2n}$ as above and restrict the action to  $\mathcal{G}_{2n}$. Such an action is automatically multiplicative. Choosing a skeleton for $(\Sigma_{2n},V_{2n})$ and extending it to a set of generators for the fundamental groupoid of $(\widehat{\Sigma_{2n}},\widehat{V_{2n}})$, we may check that the $T^2$-action thus defined lifts \eqref{eq:t2act} as desired and is the quasi-hamiltonian $T^2$-action inherited from ${{}_G \mathcal{A}_T}\times  \overline{{}_G \mathcal{A}_T}$. 
\end{proof}
\begin{rema} The double Bott-Samelson cells are obtained as quotients of the decorated double Bott-Samelson cells by the $T^2$-action \eqref{eq:t2act} (and the latter are isomorphic as varieties to the $T$-leaves of $\Gamma_{2n}$ for $G$ simply-connected or an adjoint form, as mentioned before) \cite{cludoubot}. So Theorem \ref{thm:lifact} should lead to integrating them through a quasi-Hamiltonian reduction of $\mathcal{G}_{2n}$ but to do that we need a more detailed comparison of the cluster Poisson structure used in \cite{cludoubot} with the one used here. \end{rema}   
 
\appendix
\section{Symplectic forms and groupoid structures on decorated moduli spaces}
\subsection{The symplectic form on a decorated moduli space}\label{subsec:mom map surf} Now we explain how to obtain explicit symplectic forms on the moduli spaces we use in the text. If $\Sigma$ is a closed surface, the construction described below may be adapted to recover the Atiyah-Bott symplectic structure, see \cite[\S 5]{weisymmod}. 

If $(\Sigma,V)$ is a marked surface in which $V$ meets every component of $\partial \Sigma$, we say that $(\Sigma,V)$ is a {\em fully marked surface}. Suppose for simplicity that the boundary decoration $(H,\mathcal{A})$ is given by the $\A$-orbit $H$ of the unit $1\in G^E$ and suppose that the decorated moduli space $\mathfrak{M}_G(\Sigma,V)_{H,\mathcal{A}}$ is obtained as a smooth geometric quotient as before. Then $\mathfrak{M}_G(\Sigma,V)_{H,\mathcal{A}}$ is symplectic and its symplectic form is obtained combinatorially by gluing triangles together, see \cite[\S 1.2.1]{quisur2}. This is the case for the symplectic groupoid $\mathcal{G}_k $ constructed as a decorated moduli space $\mathfrak{M}_G(\widehat{\Sigma_k},\widehat{ V_k})_{\mathcal{L}_k,\widehat{\mathcal{A}_k}   }$ in Proposition \ref{pro:intcon}, which is the key example in this paper. 

In general, the basic building block to construct $\mathfrak{M}_G(\Sigma,V)_{H,\mathcal{A}}$ corresponds to a triangle $(\Delta^2,\{v_i\}_{i=1,2,3})$ that we view here as a disc with three marked points. The corresponding representation variety is $M_{\Delta^2}:=\hom(\Pi_1({\Delta^2},\{v_i\}_{i=1,2,3}),G)$. Choose a triangulation $\mathcal{T}=(\mathcal{T}_2,\mathcal{T}_1,\mathcal{T}_0   ) $ of $\Sigma$ in which we denote by $\mathcal{T}_i$ the set of $i$-dimensional simplices; the property that we require from $\mathcal{T} $ is that $V=\partial \Sigma \cap \mathcal{T}_0$. Then the space $\hom(\Pi_1(\Sigma,V),G)$ can be represented as the quotient $M= \mathcal{M} /G^{V'}$, where $\mathcal{M}  \subset M_{\Delta^2}^{ \mathcal{T}_2  } $ is given by 
\[ \mathcal{M} =\{(\rho_T)_{T\in \mathcal{T}_2 }|\rho_T(e)=\rho_{T'}(e')^{-1} \text{ if $e=e'$ as elements of $\mathcal{T}_1 $ and $T,T'\in \mathcal{T}_2$ satisfy $T\neq T'$ } \} \]
and the action is the gauge action \eqref{eq:gauact} corresponding to the set of vertices $V'\subset \mathcal{T}_0 $ contained in the interior of $\Sigma$. Use any identification $ M_{\Delta^2}\cong G^2 $ which is consistent with the cyclic orientation of the edges in $\Delta^2$. Let $\theta^l,\theta^r$ be the respective left and right Maurer-Cartan 1-forms on $G$. Pulling back the 2-form $\omega\in \Omega^2(G^2) $ defined by 
\begin{align}  
 \omega_{(a_1,a_2)}=\frac{1}{2}\langle {a_1}^*\theta^l,{a_2}^*\theta^r \rangle\quad \forall a_1,a_2\in G ,\label{eq:polwie} 
\end{align} 
we obtain $\Omega_{\Delta^2}\in \Omega^2 \left(M_{\Delta^2}\right)$, it turns out that $\Omega_{\Delta^2}$ is independent of the chosen identification. We denote by $\iota: \mathcal{M} \subset M_{\Delta^2}^{ \mathcal{T}_2  } $ the inclusion and we let $\text{pr}_T:M_{\Delta^2}^{ \mathcal{T}_2  } \rightarrow M_{\Delta^2}$ be the projection to the factor corresponding to $T\in \mathcal{T}_2$. Then $\iota^*\sum_{T\in \mathcal{T}_2  }\text{pr}_T^* \Omega_{\Delta^2}$ descends to a 2-form $\Omega_{\Sigma,V}  \in \Omega^2({M}) $ via the quotient map $\mathcal{M} \rightarrow {M} $ \cite[Thm. 1.2]{quisur2}. The 2-form $\Omega_{\Sigma,V}$ is usually called {\em quasi-symplectic} or {\em quasi-hamiltonian}. A corollary of this discussion is that $\Omega_{\Sigma,V}$ descends to a symplectic form on the decorated moduli space $\mathfrak{M}_G(\Sigma,V)_{H,\mathcal{A}}$. 
\subsection{Groupoid structures, Morita equivalences and representation varieties}\label{app}
Here we make explicit the groupoid structures and Morita equivalences used in the text.
\subsubsection{Morita equivalences at the level of representation varieties}\label{subsec:mor} 
The groupoid structures described in \S\ref{subsec:poigromod} arise as quotients of pair groupoids, see \cite[Prop. 3.11]{poigromod}. If $(\Sigma,V)$ is a surface that is being doubled to produce $(\widehat{\Sigma},\widehat{V}  )$ as in \S\ref{subsec:poigromod}, then we have that the pair groupoid over $M:=\hom(\Pi_1(\Sigma,V),G)$ is equipped with a multiplicative moment map
\[ (\mu,\mu):M \times M \rightarrow G^E \times G^E. \]
Then, by taking the level set $(\mu,\mu)^{-1}(U)$, where
\[ U=\left( G^{E-S} \times G^{E-S}\times \prod_{e\in S } G_\Delta\right) \hookrightarrow \left(G^{E-S} \times G^{E-S} \times \prod_{e\in S} (G_e \times G_e) \right)= (G^E \times G^E); \]
we have that $\hom(\Pi_1(\widehat{\Sigma},\widehat{V}),G)$ is the quotient $(\mu,\mu)^{-1}(U)/G^L$, where $G^L$ corresponds to the diagonal gauge action \eqref{eq:gauact} at all the vertices of loops in $S$ and also at some points in $V$ which are adjacent to edges in $S$. 

The same considerations allow us to glue two marked surfaces $(\Sigma,V)$ and $(\Sigma',V')$ along a common set of edges $S\subset E$, $S\subset E'$ to produce a new marked surface $(\Sigma\cup_S \Sigma',V\cup_S V')$ in which the set of marked points is obtained by including $V\sqcup V'$ in $\Sigma\cup_S \Sigma'$ and deleting the vertices adjacent to edges in $S$ corresponding to the ones deleted before, so we consider inclusions $L\subset V$ and $L\subset V'$. Denote $\hom(\Pi_1(\Sigma',V'),G)$ by $M'$ with moment map $\mu'$ and 
\[ U_S=\left(G^{ E-S} \times G^{E'-S}  \times \prod_{e\in S } G_\Delta\right) \hookrightarrow \left(G^{ E-S} \times G^{E'-S} \times \prod_{e\in S} (G_e \times G_e) \right)= (G^E \times G^{E'}). \]
Then we have that $\hom(\Pi_1(\Sigma\cup_S \Sigma',V\cup_S V'),G)$ is obtained as the quotient $(\mu,\mu')^{-1}(U_S)/G^L$:
\begin{align}  M \times_S M':= (\mu,\mu')^{-1}(U_S) \rightarrow (\mu,\mu')^{-1}(U_S)/G^L\cong  \hom(\Pi_1(\Sigma\cup_S \Sigma',V\cup_S V'),G). \label{eq:quomodspa}\end{align} 
Note that here we are dividing by the same subgroup $G^L$ of the gauge group considered before, since we are deleting the same vertices adjacent to edges in $S$ that we deleted to produce $\hom(\Pi_1(\widehat{\Sigma},\widehat{V}),G) \cong (\mu,\mu)^{-1}(U)/G^L$. If $(\Sigma,V)=(\Sigma',V')$, then \eqref{eq:quomodspa} gives us a Lie groupoid over $\hom(\Pi_1({\Sigma},{V_0}),G)$, where $V_0$ is obtained by intersecting $V\cup_SV $ with the image of $V$ under one of the inclusions $\Sigma \subset \Sigma\cup_S \Sigma$. Otherwise, it is a space that serves as a bimodule for a Morita equivalence.
 
\begin{thm}\label{thm:morequ} Suppose that we have two inclusions $S\subset E, E'$ that induce marked surfaces by gluing as above. Then the canonical Morita equivalence between pair groupoids 
\[ \adjustbox{scale=0.85,center}{ \begin{tikzcd}
M\times M \arrow[d, shift right] \arrow[d, shift left] &  & M\times M' \arrow[lld] \arrow[rrd] &  & M'\times M' \arrow[d, shift right] \arrow[d, shift left] \\
M                                                      &  &                                    &  & M'                                                      
\end{tikzcd} } \]
descends by reduction to a Morita equivalence between Lie groupoids
\[ \adjustbox{scale=0.85,center}{\begin{tikzcd}
{\hom(\Pi_1(\widehat{\Sigma},\widehat{V}),G)} \arrow[d, shift right] \arrow[d, shift left] &  & {\hom(\Pi_1(\Sigma\cup_S \Sigma',V\cup_S V'),G)} \arrow[lld,"p"'] \arrow[rrd,"p'"] &  & {\hom(\Pi_1(\widehat{\Sigma'},\widehat{V'}),G)} \arrow[d, shift right] \arrow[d, shift left] \\
{\hom(\Pi_1({\Sigma},{V_0}),G)}                                                            &  &                                                                          &  & {\hom(\Pi_1({\Sigma'},{V_0'}),G).}                                                           
\end{tikzcd}} \]
Moreover, if $(\widehat{\Sigma},\widehat{V})$, $(\Sigma\cup_S \Sigma',V\cup_S V')$ and $(\widehat{\Sigma'},\widehat{V'})$ are fully marked surfaces as in \S\ref{subsec:mom map surf}, then their corresponding quasi-symplectic forms satisfy 
\begin{align}  \mathtt{s}^*\Omega_{\Sigma\cup_S \Sigma',V\cup_S V'}-\mathtt{t}^*\Omega_{\Sigma\cup_S \Sigma',V\cup_S V'}=\text{pr}_1^*\Omega_{\widehat{\Sigma},\widehat{V}}- \text{pr}_2^*\Omega_{\widehat{\Sigma'},\widehat{V'}}; \label{eq:equ2for}  \end{align} 
where $\mathtt{s},\mathtt{t}$ are the source and target maps of the action groupoid corresponding to the actions above:\footnotesize 
\begin{align*}  \hom(\Pi_1(\widehat{\Sigma},\widehat{V}),G) \ltimes \hom(\Pi_1(\Sigma\cup_S \Sigma',V\cup_S V'),G) \rtimes \hom(\Pi_1(\widehat{\Sigma'},\widehat{V'}),G) \rightrightarrows \hom(\Pi_1(\Sigma\cup_S \Sigma',V\cup_S V'),G)  \end{align*} \normalsize
and $\text{pr}_i$ are the corresponding projections.  
\end{thm}  
\begin{proof} The fact that the canonical Morita equivalence descends to a bimodule between the indicated groupoids is immediate. The only property that may be non obvious is the fact that $p$ and $p'$ are surjective submersions. To show surjectivity note that $p$ being surjective amounts to extending a representation of $\Pi_1(\Sigma,V_0)$ to one of $\Pi_1(\Sigma\cup_S \Sigma',V\cup_S V')$. But this is the same as taking a representation with prescribed values on $S$ and extending it to $\Pi_1(\Sigma',V')$. This is possible because we can extend $S$ to a skeleton that serves to identify $\Pi_1(\Sigma',V')$ with a free groupoid and so the representations of the latter groupoid are determined by their arbitrary values on independent generators. 

The last statement follows from assuming without loss of generality that $(\Sigma,V)$ and $(\Sigma',V')$ are fully marked surfaces themselves and then the quasi-symplectic forms satisfy 
\[ q^*\Omega_{\Sigma\cup_S \Sigma',V\cup_S V'}=(\Omega_{\Sigma,V},-\Omega_{\Sigma',V'}); \]
where $q:M \times_S M' \rightarrow \hom(\Pi_1(\Sigma\cup_S \Sigma',V\cup_S V'),G)$ is the quotient map, and similar equations hold for the other quasi-symplectic forms. As a consequence, \eqref{eq:equ2for} holds.  
\end{proof}
\begin{rema} When imposing boundary decorations, \eqref{eq:equ2for} implies that the Morita equivalences of Theorem \ref{thm:morequ} descend to commuting symplectic (or Poisson) actions on a bimodule, see \cite[Lemma 2.1]{morsym}. However, principality might not hold. \end{rema} 

\subsubsection{The groupoid structures on $\widetilde{F}^*_{2n} $ and $\mathcal{G}_{2n} $}\label{subsec:expgro}
The quasi-Poisson structures of \S\ref{subsec:quapoimodspa} may be explicitly described in terms of the choice of a skeleton for the corresponding marked surface, such a choice also makes explicit the groupoid structures on decorated moduli spaces introduced in \S\ref{sec:morequ}. 

Suppose that we take a disc with two marked points $(\Sigma_1,V_1)$. If we double it by taking two copies of $\Sigma_1$ and gluing them along two arcs, one lying in each boundary edge, we obtain $(\widehat{\Sigma_1},\widehat{V_1} ) $, a cylinder with two marked points on each boundary component, see Fig.~\ref{fig:d(g)}.
\begin{figure} 
  \adjustbox{scale=0.65}{ \begin{tikzpicture}[line cap=round,line join=round,>=stealth,x=1cm,y=1cm]
\clip(0.9225977872989136,5.054707750009231) rectangle (8.939063347504506,10.93009059971214);
\draw [rotate around={2.4681179138281935:(4.897030347083754,7.969806602197147)},line width=1.5pt] (4.897030347083754,7.969806602197147) ellipse (2.493458409507036cm and 2.442284565100449cm);
\draw [line width=1.5pt] (4.898569333327452,7.945723450432493) circle (1.0329687440559219cm);
\draw [dotted,line width=1pt,->] (4.868322580846303,10.42)-- (4.865265084346215,8.978155168850842);
\draw [dotted,line width=1pt,->] (4.864506859204151,5.527580972840095)-- (4.865542535335133,6.913282819301171);
\draw [dotted,line width=1pt,->] (4.7,10.41)-- (4.868322580846303,10.41);
\draw [dotted,line width=1pt,->]  (4.89,10.41)--(4.868322580846303,10.41);
\draw [dotted,line width=1pt,->]  (4.81,8.98)--(4.868322580846303,8.989256585177921);
\draw [dotted,line width=1pt,->]  (4.9,8.987)--(4.868322580846303,8.989256585177921);
\begin{scriptsize}
\draw [fill=black] (4.868322580846303,10.41197109892813) circle (1.5pt);
\draw[color=black] (5,10.7) node {$\widehat{v}_2'$};
\draw [fill=black] (4.864506859204151,5.527580972840095) circle (1.5pt);
\draw[color=black] (5.1,5.27) node {$\widehat{v}_2$};
\draw [fill=black] (4.865265084346215,8.978155168850842) circle (1.5pt);
\draw[color=black] (5.1,8.6) node {$\widehat{v}_1'$};
\draw [fill=black] (4.865542535335133,6.913282819301171) circle (1.5pt);
\draw[color=black] (5.1,7.1) node {$\widehat{v}_1$};
\draw[color=black] (4.1,7.9) node {$g_1$};
\draw[color=black] (6.6,7.9) node {$a_1g_2a_2^{-1}$};
\draw[color=black] (7.6,7.9) node {$g_2$};
\draw[color=black] (1.6,7.9) node {$a_1^{-1}g_1a_2$};
\draw[color=black] (5.1,6.2) node {$a_2$};
\draw[color=black] (5.1,9.6) node {$a_1$};
\end{scriptsize}
\end{tikzpicture}} \begin{tikzpicture}[line cap=round,line join=round,>=triangle 45,x=1cm,y=1cm,scale=0.4]
\clip(6.731235138195215,2.370045838934741) rectangle (14.450076564099032,10.000518546930586);
\draw [line width=1pt] (10.038525092765054,7.201928556518199) circle (2.538970219326206cm);
\draw [line width=.8pt,dashed] (8.541048363676854,9.25227809268628)-- (7.707309555030952,6.196043925334053);
\draw [line width=.8pt,dashed] (7.707309555030952,6.196043925334053)-- (9.951549208037106,4.664448513226991);
\draw [line width=.8pt,dashed] (9.951549208037106,4.664448513226991)-- (12.400823407913643,6.271382241243403);
\draw [line width=.8pt,dashed] (12.400823407913643,6.271382241243403)-- (11.466682015570012,9.301151645775183);
\begin{scriptsize}
\draw [fill=black] (11.466682015570012,9.301151645775183) circle (1.5pt);
\draw[color=black] (11.748194840341337,9.473334159621421) node {$v_5$};
\draw [fill=black] (8.541048363676854,9.25227809268628) circle (1.5pt);
\draw[color=black] (8.128123976674797,9.437471276131001) node {${v}_1$};
\draw [fill=black] (9.951549208037106,4.664448513226991) circle (1.5pt);
\draw[color=black] (9.88415704638585,5.254610839828577) node {$v_3$};
\draw [fill=black] (7.707309555030952,6.196043925334053) circle (1.5pt);
\draw[color=black] (7.133216184507002,6.3102278357664305) node {${v}_2$};
\draw [fill=black] (12.400823407913643,6.271382241243403) circle (1.5pt);
\draw[color=black] (12.952771651846273,6.274364952276011) node {$v_4$};
\end{scriptsize}
\end{tikzpicture}
    \caption{The edges labelled by $\{g_1,g_2,a_1,a_2\}$ constitute a skeleton of the cylinder and $\mu$ is given by restricting a representation to the boundary; to the right a skeleton for $\Sigma_{0,4}$}
    \label{fig:d(g)}
\end{figure}
Using the generators displayed in Fig.~\ref{fig:d(g)}, the groupoid structure on $\hom(\Pi_1(\widehat{\Sigma_1},\widehat{V_1} ),G) \rightrightarrows \hom(\Pi_1({\Sigma_1},{V_1} ),G)$ becomes the product of the pair groupoid $G \times G \rightrightarrows G$ and two copies of $G$ as a Lie group $\mathcal{E} :=((G \times G)\times G^2)\rightrightarrows G$, the moment map $\mu$ is then given by:
\begin{align}
    &\mu(a_1,a_2,g_1,g_2)=((a_1g_2a_2^{-1},g_1),(a_1^{-1}g_1a_2,g_2))\quad  \forall (a_1,a_2,g_1,g_2)\in \mathcal{E} ;\label{eq:mommapgro}
\end{align} 
In Fig.~\ref{fig:d(g)}, we may identify $a_i,g_i$ with the values of a representation $\rho\in \hom(\pi_1(\widehat{\Sigma_1},\widehat{V_1}),G)$ on the skeleton therein displayed; the source and target maps are given then by applying $\rho$ to $a_i$ as in the figure.
 
To obtain the groupoid structures of \eqref{eq:dousurdousym} and \eqref{eq:dousurdual} we may just take a skeleton of the surface that is being doubled and then we take the fibred product of as many copies of the groupoid $\mathcal{E} $ as there are edges in the skeleton, with respect to the moment maps and the incidence relations prescribed by the skeleton. We start with the case of \eqref{eq:dousurdual} as \eqref{eq:dousurdousym} can be expressed in terms of it, as we shall see in Proposition \ref{pro:slim} below.
\begin{prop}\label{pro:expduapoi} We have that $\widetilde{F}^*_{2n} \rightrightarrows F_n$ is a complex algebraic Poisson groupoid, which is isomorphic to the quotient of an action groupoid:
\[ (G^* \times N_+^{n-1}) \times_{B_+^{n}} G^n     \rightrightarrows F_n; \] 
where the space of arrows is an associated bundle to the principal $B_+^n$-bundle $G^n \rightarrow F_n$ and the multiplication is specified in \eqref{eq:grostrdua} below.  \end{prop} 
\begin{proof} We will provide a description of the desired groupoid structure as the quotient by a multiplicative action of $B_+^{2n} \rightrightarrows B_+^n$ (regarded as a pair groupoid $B_+^n \times B_+^n \rightrightarrows B_+^n $) on an action groupoid:
\begin{align}  \left[(G^* \times \widetilde{B}_+^{n-1} \times B_+) \ltimes  G^n  \right] \rightrightarrows G^n. \label{eq:duagroske}\end{align}  We start with the following skeleton of $(\Sigma_{0,n},V_{0,n})$: we take one edge between each pair of consecutive marked points except for $v_{n+1}$ and $v_1$ as in the disc displayed in Fig.~\ref{fig:d(g)}. The groupoid \eqref{eq:duagroske} is then obtained as the fibred product of $n$ copies of $\mathcal{E} $ and the $n$ subgroups $L_{v_i}\subset G \times G$ given by the decoration \eqref{eq:decduagro}. We only give the final simplified description of this fibred product below.

Embed the edges of the skeleton of $(\Sigma_{0,n},V_{0,n})$ described above using one of the inclusions $(\Sigma_{0,n},V_{0,n}) \subset (\widehat{\Sigma_{0,n}},\widehat{V_{0,n}})  $. Taking the union of these edges with all the boundary edges in $\widehat{\Sigma_{0,n}} $ except for the one decorated with $G$, we obtain a skeleton for $(\widehat{\Sigma_{0,n}},\widehat{V_{0,n}})$, see Fig.~\ref{fig:skedual}.
To introduce notation, we let $(z_i,z_i')\in \widetilde{B_+} $ be the values of a representation on the boundary edges of a boundary component decorated with $\widetilde{B_+}$, similarly, $(x,y)\in G^*$ correspond to the boundary component decorated with $G^*=B_+ \times_T B_-\subset G \times G$ and $(g,b)\in G \times B_+$ corresponds to the remaining boundary component. Then the action map of \eqref{eq:duagroske} is given by 
\[ ((x,y),(z_i,z_i'),b) \cdot (g_1,\dots, g_n) =(h_1,\dots,h_n ); \]
where the components of $[h_1,\dots,h_n]$ are as below (their geometric meaning is illustrated by Fig.~\ref{fig:skedual}):
\begin{equation}  h_i=\begin{cases} &xg_1z_1^{-1}, \quad \text{if $i=1 $}  \\
&z_{i-1}'g_iz_i^{-1}, \quad \text{if $1<i\leq n-1$,} \\
&z_{n-1}'g_{n}b^{-1}, \quad \text{if $i=n$;} \end{cases}  \label{eq:grostrdua} \end{equation} 
for $n=1$, we put $h_1=xg_1b^{-1}$. This formula specifies the groupoid structure of \eqref{eq:duagroske} (which is isomorphic to a subgroupoid of \eqref{eq:dousurdual}).
\begin{figure}\centering 
\begin{tikzpicture}[line cap=round, line join=round, >=stealth, scale=0.6]

\coordinate (TL) at (0, 3);
\coordinate (BL) at (0, -3);
\coordinate (C1top) at (3, 0.8);
\coordinate (C1bot) at (3, -0.8);
\coordinate (C2top) at (7, 0.8);
\coordinate (C2bot) at (7, -0.8);

\draw [line width=1pt] (TL) -- (BL) node[midway, left=5pt] {};
\draw [line width=1pt] (TL) .. controls (12, 3) and (12, -3) .. (BL) node[pos=0.5, right=5pt] {$b$};

\draw [line width=1pt, ->] (C1bot) to[bend left=90] node[pos=0.5, left=2pt] {$y$} (C1top);
\draw [line width=1pt, ->] (C1bot) to[bend right=90] node[pos=0.5, right=2pt] {$x$} (C1top);

\draw [line width=1pt, ->] (C2bot) to[bend left=90] node[pos=0.5, left=2pt] {$z_1$} (C2top);
\draw [line width=1pt, ->] (C2bot) to[bend right=90] node[pos=0.5, right=2pt] {$z_1'$} (C2top);

\draw [line width=0.4pt,dashed,<-] (C1top) to[bend left=25] node[pos=0.5, below=1pt] {$h_1$} (C2top);
\draw [line width=0.4pt,dashed,<-] (C1bot) to[bend right=25] node[pos=0.5, above=1pt] {$g_1$} (C2bot);

\draw [line width=0.4pt,dashed,->] (TL) to[bend left=20] node[pos=0.4, below=1pt] {$h_2$} (C2top);
\draw [line width=0.4pt,dashed,->] (BL) to[bend right=20] node[pos=0.4, above=2pt] {$g_{2}$} (C2bot);

\begin{scriptsize}
\fill (TL) circle (1.5pt) node[left=3pt] {$\widehat{v}_3'$};
\fill (BL) circle (1.5pt) node[left=3pt] {$\widehat{v}_3$};
\fill (C1top) circle (1.5pt) node[above=2pt] {$\widehat{v}_1'$};
\fill (C1bot) circle (1.5pt) node[below=2pt] {$\widehat{v}_1$};
\fill (C2top) circle (1.5pt) node[above right] {$\widehat{v}_2'$};
\fill (C2bot) circle (1.5pt) node[below right] {$\widehat{v}_2$};
\end{scriptsize}

\end{tikzpicture}
\caption{The generators for $\Pi_1(\widehat{\Sigma_{0,2}},\widehat{V_{0,2}}) $ displayed above illustrate the groupoid structure $\widetilde{F}_{4}^*\rightrightarrows F_2 $}
\label{fig:skedual}
\end{figure}
It also follows from this description that \eqref{eq:duagroske} contains the action groupoid $B_+^n \ltimes G^n \rightrightarrows G^n$ induced by the gauge action \eqref{eq:gauact} on $\hom(\Pi_1(\Sigma_{0,n},V_{0,n}),G)$ restricted to the subgroup of $G^V$ determined by the decoration, which is exactly $B_+^n$; in fact, we have the inclusion $B_+^n \subset G^* \times \widetilde{B}_+^{n-1} \times B_+$ given by $(b_i)_i\mapsto ((1,1),(b_i,b_i)_{i<n},b_n)$. It follows that $\widetilde{F}_{2n}^*$ is isomorphic to the quotient of \eqref{eq:duagroske} by the action of $B_+^{2n} $ defined by right and left translation by elements in $B_+^n \ltimes G^n \rightrightarrows G^n$: 
\[ (b_i,c_i)_i \cdot u=\mathtt{m}((b_i,\mathtt{t}(u)),\mathtt{m}( u,\mathtt{i} (c_i,\mathtt{s}(u)))) . \]
As a consequence, using the map $p:\widetilde{B_+} \rightarrow N_+$ determined by $(a,b)\mapsto a^{-1}b$, we have that the groupoid structure on $\widetilde{F}_{2n}^* \rightrightarrows F_n$ is the unique one such that the quotient map $q_{\widetilde{F}_{2n}^* }: (G^* \times \widetilde{B}_+^{n-1} \times B_+) \times  G^n \rightarrow (G^* \times N_+^{n-1}) \times_{B_+^{n}} G^n     \rightrightarrows F_n$
\[ (((x,y),(z_i,z_i'),b),g_1\dots g_n) \mapsto ((x,y),[p(z_i,z_i'),g_1\dots g_n]) \]
is a groupoid morphism. Denote by $\widehat{v}_i,\widehat{v}_i'$ the two images of $v_i\in V_{0,n}$ in $\widehat{\Sigma_{0,n}}$ under the natural inclusions ${\Sigma_{0,n}}\subset \widehat{\Sigma_{0,n}}$, see Fig.~\ref{fig:skedual}. The quotient $q_{\widetilde{F}_{2n}^* }$ of the gauge action of $B_+^{2n}$ may be described by stages, first by letting a copy of $B_+^n$ act at the vertices $\widehat{v}_i'$ and then the remaining copy of $B_+^n$ acts at the vertices $\widehat{v}_i$. Hence the quotient $q_{\widetilde{F}_{2n}^* }$ may be viewed as the composite
\[ \begin{tikzcd}
G^*\times \widetilde{B}_+^{n-1} \times B_+ \times G^n  \arrow[rr] &  & G^*\times N_+^{n-1}\times G^n \arrow[rr] &  & \widetilde{F}_{2n}^*\cong (G^*\times N_+^{n-1})\times_{B_+^n} G^n
\end{tikzcd} \]
of two quotient maps for associated bundles: the first one corresponding to the homogeneous space quotient map $G^*\times \widetilde{B}_+^{n-1} \times B_+ \rightarrow G^*\times N_+^{n-1}$ and the trivial action on $G^n$; the second one being associated to the principal bundle $G^n \rightarrow F_n$ and the restricted gauge action of $B_+^n$ on $G^*\times N_+^{n-1}$. Therefore, $q_{\widetilde{F}_{2n}^*}$ is a geometric quotient \cite[Prop. 2.1.1.7]{gitdec} (the statement only mentions the existence of a categorical quotient but the proof shows it is geometric as well). Since the structure sheaf on the quotient is obtained from the invariant functions on $G^*\times \widetilde{B}_+^{n-1} \times B_+ \times G^n$, which inherit a Poisson structure from the quasi-Poisson formalism, we obtain the result. \end{proof}
\begin{rema} Note that, in order to apply the construction of \cite{gitdec} to our situation, we have to promote $G^*\times \widetilde{B}_+^{n-1} \times B_+ \times G^n$ to an affine scheme and then its quotient as above is obtained by gluing the quasi-affine schemes determined by a local trivialization of an algebraic principal bundle (in our case, such local trivializations exist in the Zariski topology). Then $(G^*\times N_+^{n-1})\times_{B_+^n} G^n$, as a complex manifold, is obtained as the analytification of the scheme resulting from that gluing. \end{rema} 
\begin{coro}\label{cor:gensch} Any generalized Schubert cell $\mathcal{O}^{\mathbf{u} } \subset F_n\subset \widetilde{F}_{2n}$ is contained in a single symplectic leaf of $\widetilde{F}_{2n}$ and in a symplectic subgroupoid of the form $\mathcal{G}\cdot^V \overline{\mathbf{u}} \subset \widetilde{F}_{2n} $. \end{coro} 
\begin{proof} Here we denote by $\overline{\mathbf{u}}$ the image of $\mathbf{u}\in W^n$ in $F_n$ under the quotient map $G^n \rightarrow F_n$. Formula \eqref{eq:grostrdua} implies that $\mathcal{O}^{\mathbf{u} }$ is a $\widetilde{F}_{2n}^*$-orbit and it is connected, since $\widetilde{F}_{2n}^*$ is source-connected (this follows from the connectivity of $G^*,N_+,B_+$). So we have that $\mathcal{O}^{\mathbf{u} }\subset \mathcal{G}\cdot^V \overline{\mathbf{u}}$, but $\mathcal{G}\cdot^V \overline{\mathbf{u}}$ is a symplectic groupoid that contains the symplectic leaf through $\overline{\mathbf{u}}$, see Proposition \ref{cor:symorb}. Therefore, since $\mathcal{O}^{\mathbf{u} }$ is connected, it is contained in the connected component of $\mathcal{G}\cdot^V \overline{\mathbf{u}}$ which passes through $\mathbf{u}$, which is such a symplectic leaf.  \end{proof} 
  
\begin{prop}\label{pro:slim} The following statements hold. 
\begin{enumerate} \item The symplectic groupoid $\mathcal{G}_k$ of Proposition \ref{pro:intcon} is complex algebraic.
\item The symplectic double groupoid $\mathcal{G}_{2n}$ of Proposition \ref{pro:symdou} is a slim double groupoid which is a complex algebraic subvariety of    
\[ \widetilde{F}_{2n} \times \widetilde{F}_{2n}^* \times \widetilde{F}_{2n}^* \times \widetilde{F}_{2n}.    \]
\item The actions of $\mathcal{G}_{2m}$ and $\mathcal{G}_{2n}$ in the Morita equivalence of Theorem \ref{thm:symdoumor} are algebraic.
\end{enumerate} 
 \end{prop}
\begin{proof} We start with item (1). Suppose that $k=m+n$, using the notation of Theorem \ref{thm:symdoumor}, we shall show that 
\[ \Theta:=(\mathtt{t},q,p,\mathtt{s}  ):\mathcal{G}_{m+n}\rightarrow Z:= \widetilde{F}_{m+n} \times \widetilde{F}_{2n}^* \times \widetilde{F}_{2m}^* \times \widetilde{F}_{m+n}     \]
is an injective map with image a complex algebraic subvariety. We may describe $\widehat{\Sigma}_{m+n}$ as a disc with $m+n$ holes and represent it in such a way that the reflection symmetry with respect to the horizontal axis determines the vertical groupoid structure on $\mathcal{G}_{2n}$ and reflection with respect to the vertical axis induces the restriction maps $p$ and $q$, see Fig.~\ref{fig:dousurdec}. Then we use skeleta for $\Sigma_{0,m}$ and $\Sigma_{0,n}$ with edges given by joining each vertex $v_i$ with $v_1$ for $i>1$, as in the proof of Theorem \ref{pro:conflaqpoi}. By embedding these skeleta into $\Sigma_{m+n}$, we obtain a skeleton $\mathfrak{S} $ with edges $\{\mathbf{e}_i \}_{i=1...m+n}$, the two inclusions of $\mathbf{e}_i$ in $\widehat{\Sigma_{m+n}}$ shall be denoted by $\mathbf{e}_i $ and $\mathbf{e}_i'$ for simplicity. Denote by $\iota_1:{\widehat{\Sigma_{0,m}}} \hookrightarrow \widehat{\Sigma_{m+n}}$ and $\iota_2:{\widehat{\Sigma_{0,n}}} \hookrightarrow \widehat{\Sigma_{m+n}}$ the two inclusions. Let $(x_i,y_i)\in G^*$ be the values of a representation $\rho\in \Pi_1(\widehat{\Sigma_{m+n}},\widehat{V_{m+n}}) $ on the boundary edges lying over $\iota_i(v_1)$, and likewise let $(z_i,z_i')\in \widetilde{B}_+ $ be the values of $\rho$ on the other boundary edges lying in the interior of $\widehat{\Sigma}_{m+n}$. Finally, let $(b_1,b_2)\in \widetilde{B}_+$ be the values of $\rho$ on the outer boundary circle, see Fig.~\ref{fig:dousurdec}. We may identify a representation with its values on the indicated generators to obtain that  
\begin{align*}  &\Theta(\rho)=([h_i]_{i=1...n+m},q(\rho),p(\rho),[g_i]_{i=1...n+m}),\\
& q(\rho)=((x_2,y_2),[(z_i,z_i')_{i=m...m+n-2},(g_i)_{i>m}]),\\ 
&p(\rho)=((x_1,y_1),[(z_i,z_i')_{i=1...m-1},(g_i)_{i\leq m}]). 
 \end{align*}   
Note that $\mathtt{t}(q(\rho))= [h_i]_{i>m}$ and $\mathtt{t}(p(\rho))=[h_i]_{i\leq m}$, where $g_i=\rho(\mathbf{e}_i)$ and $h_i=\rho(\mathbf{e}_i')$ with respect to the groupoid structure \eqref{eq:duagroske}, see Fig.~\ref{fig:dousurdec}. So the map $\Theta$ is clearly an embedding. Conversely, $([h_i]_{i=1...m+n},\rho_2,\rho_1,[g_i]_{i...m+n})$ lies in the image of $\Theta$ if and only if 
\begin{equation} \begin{aligned} &q_0([h_i]_{i=1...m+n})=\mathtt{t}_{\widetilde{F}_{2n}^* }(\rho_2), \quad p_0([g_i]_{i=1...m+n})=\mathtt{s}_{\widetilde{F}_{2m}^*}(\rho_1),\\
& q_0([g_i]_{i=1...m+n})=\mathtt{s}_{\widetilde{F}_{2n}^* }(\rho_2), \quad p_0([h_i]_{i=1...m+n})=\mathtt{t}_{\widetilde{F}_{2m}^*}(\rho_1),\\
&f_1:=y_1g_1g_{m+1}^{-1}(h_1h_{m+1}^{-1}y_2)^{-1}=1,\quad f_{2}=(b_1,b_2)\in \widetilde{B}_+; \end{aligned} \label{eq:eqs} \end{equation}   
geometrically, $f_2\in \widetilde{B}_+$ is the decoration of the outer boundary circle, whereas $f_1=1$ means the triviality of the central loop in Fig.~\ref{fig:dousurdec}.
 
Denote by $q_{\widetilde{F}_l }:G^l \rightarrow G^l/B_+^{l-1}$ and $q_{\widetilde{F}_{2l}^* }:G^*\times \widetilde{B}_+^{l-1} \times B_+ \times G^l  \rightarrow (G^* \times N_+^{l-1}) \times_{B_+^l} G^l$ the natural quotient maps. Equations \eqref{eq:eqs} can be naturally lifted to define a complex algebraic subvariety 
\[ X\subset Y:=G^{m+n} \times \left[G^*\times \widetilde{B}_+^{n-1} \times B_+ \times G^n\right] \times \left[(G^*\times \widetilde{B}_+^{m-1} \times B_+ \times G^m) \right]\times G^{m+n} \]
but the quotient map $Q:Y \rightarrow Z$ is the product of maps of the form of $q_{\widetilde{F}_l }$ and $q_{\widetilde{F}_{2l}^* }$ and hence it is a geometric quotient. In fact, $Q$ is the composite of geometric quotients just like $q_{\widetilde{F}_{2l}^*}$, as discussed in Proposition \ref{pro:expduapoi}. So the Zariski closed invariant subset $X\subset Y$ projects to a Zariski closed subset of $Z$, which is exactly the image of $\Theta$.  

Using this description, we can see that the structure maps of $\mathcal{G}_k$ are algebraic, being determined by the ones of the groupoids $\widetilde{F}_{2l}^*$ as in the second formula of \eqref{eq:luweimul}. For the same reason, (2) and (3) hold. \end{proof}

\begin{figure}[H] 
\begin{tikzpicture}[>=stealth]

    \draw (0,0) ellipse (4cm and 2cm);
    \fill (0, 2) circle (1pt) node[above] {$\iota_1(\widehat{v}_2')=\iota_2(\widehat{v}_3')$}; 
    \fill (0, -2) circle (1pt) node[below] {$\iota_1(\widehat{v}_2)=\iota_2(\widehat{v}_3)$};
    
    \draw[dotted] (0, 2) -- (0, -2);
    
    \node at (-3, 1.6) {$b_1$};
    \node at (3, 1.6) {$b_2$};

    \fill (-1.5, 0.6) circle (1pt) node[left] {$\iota_1(\widehat{v}_1')$};
    \fill (-1.5, -0.6) circle (1pt) node[left] {$\iota_1(\widehat{v}_1)$};
    
    \draw[->] (-1.5, -0.6) to[bend left=80] node[left, scale=0.8] {$x_{1}$} (-1.5, 0.6);
    \draw[->] (-1.5, -0.6) to[bend right=80] node[right, scale=0.8] {$y_{1}$} (-1.5, 0.6);

    
    \fill (1.2, 0.6) circle (1pt) node[left] {$\iota_2(\widehat{v}_1')$};
    \fill (1.2, -0.6) circle (1pt) node[left] {$\iota_2(\widehat{v}_1)$};
    
    \draw[->] (1.2, -0.6) to[bend left=80] node[left, scale=0.8] {$y_{2}$} (1.2, 0.6);
    \draw[->] (1.2, -0.6) to[bend right=80] node[left, scale=0.8] {$x_{2}$} (1.2, 0.6);
    
    \fill (2.4, 0.6) circle (1pt) node[above right] {$\iota_2(\widehat{v}_2')$};
    \fill (2.4, -0.6) circle (1pt) node[below right] {$\iota_2(\widehat{v}_2)$};
    
    \draw[->] (2.4, -0.6) to[bend left=80] node[right, scale=0.8] {$z_{1}$} (2.4, 0.6);
    \draw[->] (2.4, -0.6) to[bend right=80] node[right, scale=0.8] {$z_{1}'$} (2.4, 0.6);

    \draw[dashed,<-] (1.2, 0.6) to[out=15, in=165] node[above] {$h_3$} (2.4, 0.6);
    \draw[dashed,<-] (1.2, -0.6) to[out=-15, in=-165] node[below] {$g_3$} (2.4, -0.6);

    \draw[dashed,->] (0, 2) to[out=180, in=60] node[pos=0.6, left] {$h_1$} (-1.5, 0.6);
    \draw[dashed,->] (0, 2) to[out=0, in=120] node[pos=0.6, right] {$h_2$} (1.2, 0.6);
    
    \draw[dashed,->] (0, -2) to[out=180, in=-60] node[pos=0.6, left] {$g_1$} (-1.5, -0.6);
    \draw[dashed,->] (0, -2) to[out=0, in=-120] node[pos=0.6, right] {$g_2$} (1.2, -0.6);

\end{tikzpicture}
\caption{Generators for the fundamental groupoid of $\widehat{\Sigma_{m+n}}$ when $m=1$, $n=2$}
    \label{fig:dousurdec}
\end{figure}

\printbibliography\footnotesize

\end{document}